\documentclass[a4paper,10pt]{amsart}

\usepackage{a4wide}
\usepackage[latin1]{inputenc} 
\usepackage[T1]{fontenc}

\usepackage{amsfonts}
\usepackage{amsmath,amssymb,amsthm,stmaryrd}

\usepackage{mathrsfs}
\usepackage{esint}

\usepackage{fancyhdr}
\usepackage{verbatim}

\usepackage{graphicx}
\usepackage{ifpdf}
\ifpdf
\fi 

\DeclareMathOperator{\diam}{diam}

\DeclareMathOperator{\pr}{pr}

\newcommand{\ud}[0]{\,\mathrm{d}}

\newcommand{\eps}[0]{\varepsilon}

\providecommand{\abs}[1]{\lvert#1\rvert} 

\newcommand{\Norm}[2]{\|#1\|_{#2}}

\newcommand{\BMO}[0]{\operatorname{BMO}}

\newcommand{\loc}[0]{\operatorname{loc}}
\newcommand{\lab}[0]{\operatorname{label}}

\newcommand{\Exp}[0]{\mathbb{E}}
\newcommand{\prob}[0]{\mathbb{P}}

\newcommand{\R}{\mathbb{R}}

\newcommand{\Z}{\mathbb{Z}}
\newcommand{\N}{\mathbb{N}}

\swapnumbers
\numberwithin{equation}{section}

\makeatletter
  \let\c@equation\c@subsection
\makeatother

\theoremstyle{plain}
\newtheorem{theorem}[subsection]{Theorem}
\newtheorem{lemma}[subsection]{Lemma}
\newtheorem{corollary}[subsection]{Corollary}
\newtheorem{proposition}[subsection]{Proposition}

\theoremstyle{definition}
\newtheorem{definition}[subsection]{Definition}

\theoremstyle{remark}
\newtheorem{remark}[subsection]{Remark}

\makeatletter
\@namedef{subjclassname@2010}{%
  \textup{2010} Mathematics Subject Classification}
\makeatother

\title{Systems of dyadic cubes in a doubling metric space}
\author{Tuomas Hyt\"onen}
\author{Anna Kairema}
\address{Department of Mathematics and Statistics, P.O.B. 68 (Gustaf H\"allstr\"omin katu 2), FI-00014 University of Helsinki, Finland}
\email{tuomas.hytonen@helsinki.fi, anna.kairema@helsinki.fi}

\thanks{The authors are supported by the Academy of Finland, grants 130166, 133264 and 218148.}

\begin{document}
\maketitle

\begin{abstract}
A number of recent results in Euclidean Harmonic Analysis have exploited several adjacent systems of dyadic cubes, instead of just one fixed system. In this paper, we extend such constructions to general spaces of homogeneous type, making these tools available for Analysis on metric spaces. The results include a new (non-random) construction of boundedly many adjacent dyadic systems with useful covering properties, and a streamlined version of the random construction recently devised by H.~Martikainen and the first author. We illustrate the usefulness of these constructions with applications to weighted inequalities and the BMO space; further applications will appear in forthcoming work.
\end{abstract}

\section{Introduction}

The standard system of dyadic cubes,
\begin{equation*}
  \mathscr{D}:=\{2^{-k}\big([0,1)^n+m\big):k\in\Z,m\in\Z^n\},
\end{equation*}
plays an indispensable role in Harmonic Analysis on the Euclidean space $\R^n$. The fundamental properties of these cubes are that any two of them are either disjoint or one is contained in the other, and that the cubes of a given size partition all space. Also, the cubes are not too far away from balls, which are usually more natural objects from the geometric point of view.

Accordingly, there has been interest in constructing analogues of dyadic cubes in more general metric spaces, to provide tools for Analysis in such settings. The first results in this direction, as far as we know, are due to G.~David \cite{David88}, Appendix~A; see also \cite{David91}, Appendix~I. A more comprehensive construction was provided by M.~Christ \cite{Christ90}, in the full generality of Coifman--Weiss spaces of homogeneous type \cite{CW71}, and this has become the standard reference on the topic. In addition to the basic geometric properties expected from the cubes, Christ also obtained a certain smallness of the boundary condition (in terms of an underlying doubling measure), which has turned out useful in applications to singular integrals.  A more elementary construction, without addressing the boundary control but nevertheless sufficient for many purposes, was provided by E. Sawyer and R. L. Wheeden \cite{SW92}. Some further variants have been considered by other authors \cite{ABI,KRS}.

Meanwhile, new developments in $\R^n$ have seen the need to consider not just one but several adjacent dyadic systems. Two types of constructions are of particular interest here.  On the one hand, there are the \emph{random dyadic systems} due to F.~Nazarov, S.~Treil and A.~Volberg \cite{NTV:Cauchy}, Section~4,
\begin{equation*}
  \mathscr{D}(\omega):=\Big\{2^{-k}\big([0,1)^n+m\big)+\sum_{j>k}2^{-j}\omega_j:k\in\Z,m\in\Z^n\Big\},\qquad
  \omega=(\omega_j)_{j\in\Z}\in\Omega:=(\{0,1\}^n)^{\Z},
\end{equation*}
where $\Omega$ is equipped with the natural product probability measure. The randomization provides a powerful way of controlling edge effects, even when the space is equipped with a non-doubling measure, by proving that any given point $x\in\R^n$ has a small probability of ending up close to the boundary of a randomly chosen cube. These random dyadic systems have been instrumental in the development of the non-doubling theory of singular integrals \cite{NTV:Cauchy,NTV:Tb} and its applications to analytic capacity \cite{Tolsa:Painleve,Volberg:03} as well as sharp one-weight \cite{Hytonen:A2,HPTV} and two-weight \cite{LSUT:Hilbert,NTV:Hilbert} inequalities for classical singular integrals. A version of such random cubes in metric spaces, starting from Christ's construction, was recently obtained by the first author and H.~Martikainen \cite{HM09}, to study singular integrals in metric spaces with non-doubling measures.

On the other hand, there is a \emph{non-random choice} of just boundedly many dyadic systems, say
\begin{equation*}
  \mathscr{D}^{t}:=\{2^{-k}\big([0,1)^n+m+(-1)^k t\big):k\in\Z,m\in\Z^n\},\qquad t\in\{0,\tfrac13,\tfrac23\}^n,
\end{equation*}
which have the following useful property: for every ball $B$, there exists a cube $Q$ in at least one of the $\mathscr{D}^{t}$ such that $B\subseteq\tfrac{9}{10} Q$ while $\diam(Q)\leq C\diam(B)$. These adjacent systems have been exploited, e.g., in the work of C.~Muscalu, T.~Tao and C.~Thiele \cite{MTT:02,MTT:03} on multi-linear operators, and very recently by M.~Lacey, E.~Sawyer and I.~Uriarte-Tuero \cite{LSUT:08} on two-weight inequalities. We are not aware of the precise original occurrence of this latter set of systems: Lacey et al. (\cite{LSUT:08}, Section~2.2) attribute them to J.~Garnett and P.~Jones; Muscalu et al. (\cite{MTT:02}, Section~5) to M.~Christ, at least what comes to the observation on $B\subseteq\tfrac{9}{10} Q$. T.~Mei \cite{Mei03} has shown that a similar conclusion can be obtained with just $n+1$ (rather than $3^n$) cleverly chosen systems $\mathscr{D}^{t}$.

The goals of the present paper are two-fold. First, we recall and streamline the construction of the Christ-type dyadic cubes in a metric space, including the recent randomized version $\mathscr{D}(\omega)$ from~\cite{HM09}. Second, we provide a metric space version of the non-random choice of boundedly many dyadic systems $\mathscr{D}^t$, with the property that any $B$ is contained in some $Q\in\bigcup_t\mathscr{D}^t$ with $\diam(Q)\leq C\diam(B)$, which is a completely new result. We even combine the two constructions, yielding a random family of adjacent dyadic systems $\mathscr{D}^t(\omega)$.

We have strived for a reasonably comprehensive and transparent presentation, including some results which could be found elsewhere. In the hope of making the paper a useful reference, we have tried to make the statements of our theorems easily applicable as ``black boxes'', but also paid attention to the details of the proofs. As in $\R^n$, where the open ($2^{-k}((0,1)^n+m)$), half-open ($2^{-k}([0,1)^n+m)$) and closed ($2^{-k}([0,1]^n+m)$) dyadic cubes each serve their different purpose, we also present (as in \cite{HM09}) a unified construction of three related systems of open, half-open and closed cubes. While this may not be the absolute shortest route to the half-open cubes alone (for which one should probably consult the recent paper of A.~K\"aenm\"aki, T.~Rajala and V.~Suomala~\cite{KRS}), we find the properties of the open and closed cubes proven on the way to be useful as well.

The basic idea behind the construction of several adjacent dyadic systems is from \cite{HM09}: New centre points for the cubes of sidelength $\delta^k$ (where $\delta>0$ is a small parameter playing the role of the constant $\tfrac12$ in the classical Euclidean dyadic system) are chosen among the old centre points of the (one level smaller) cubes of sidelength $\delta^{k+1}$. In the non-probabilistic selection, instead of doing this randomly, we need to do it in a ``clever'' way. The basic \emph{conflict} to avoid is two new centres getting too close to each other. This is achieved by equipping the points with suitable \emph{labels}, which help us in avoiding these conflicts. As it turns out, this philosophy can also be used to simplify the original random construction from \cite{HM09}, where originally the conflicts were first allowed among the new centre points, and yet another selection process was needed to remove some of them, thereby yielding the final points. This simplification already proved useful in the consideration of vector-valued singular integrals by the random cubes method by Martikainen~\cite{Henri:10}, and it is expected to be of interest elsewhere. A particular feature of the new selection process is a natural one-to-one correspondence between the old and new cubes, which was not present, when some of the new centres were first removed; thus the original cubes may be used as an index set for the new cubes, a technical property which was much exploited in some of the recent Euclidean applications \cite{Hytonen:A2,HPTV}.

As an illustration of the use of the new adjacent dyadic systems, we provide easy extensions  of two results in Euclidean harmonic analysis to metric spaces $X$ with a doubling measure $\mu$: First, Buckley's theorem \cite{Buckley:93} on the sharp weighted norm of the Hardy--Littlewood maximal operator,
\begin{equation*}
  \Norm{Mf}{L^p(w)}\leq C\Norm{w}{A_p}^{1/(p-1)}\Norm{f}{L^p(w)},
  \qquad\Norm{w}{A_p}:=\sup_B\Big(\fint_B w\ud\mu\Big)\Big(\fint_B w^{-1/(p-1)}\ud\mu\Big)^{p-1},
\end{equation*}
where the supremum is over all metric balls in $X$, and $C$ only depends on $X$, $\mu$ and $p\in(1,\infty)$. Second, the representation of $\BMO(\mu)$ as an intersection of finitely many dyadic BMO spaces, extending the Euclidean result in \cite{Mei03}.

In extending Buckley's theorem, we follow the Euclidean approach due to Lerner~\cite{Lerner:08}. A noteworthy feature of our argument is the circumvention of the use of the Besicovitch covering theorem, an essentially Euclidean device used in Lerner's original proof, by the trivial covering properties exhibited by the adjacent dyadic systems. We believe that this displays a more general principle of avoiding the Besicovitch theorem and thereby allowing extensions of other Euclidean results to metric spaces.

 Note that only these applications, but not the construction of the cubes as such, depends on the existence of a doubling measure $\mu$ on~$X$; for the cubes, we only need the weaker geometric doubling property that any ball of radius $r$ can be covered by at most $A_1$ (a fixed constant) balls of radius~$\tfrac12 r$. Further applications will be considered in a forthcoming paper by the second author.


\section{Definition and construction of a dyadic system}

\subsection{Set-up}
Let $\rho$ be a quasi-metric on the space $X$, i.e., it satisfies the axioms of a metric except for the triangle inequality, which is assumed in the weaker form
\begin{equation*}
  \rho(x,y)\leq A_0\big(\rho(x,z)+\rho(z,y)\big)
\end{equation*}
with a constant $A_0\geq 1$.
The quasi-metric space $(X,\rho)$ is assumed to have the following \emph{(geometric) doubling property}: There exists a positive integer $A_1\in \N$ such that for every $x\in X$ and for every $r>0$, the ball $B(x,r):=\{y\in X:\rho(y,x)<r\}$ can be covered by at most $A_1$ balls $B(x_i,r/2)$.

Until further notice, no other properties of the quasi-metric space $(X,\rho)$ will be required; in particular, we do not assume any measurability of $X$. Some of the arguments are valid even without the assumption of geometric doubling.

Set $a_1:=\log_2 A_1$. The following properties are easy to check; cf. \cite[Lemmas 2.3 and 2.5]{Framework}:
\begin{enumerate}
  \item Any ball $B(x,r)$ can be covered by at most $A_1\delta^{-a_1}$ balls of radius $\delta r$ for any $\delta\in(0,1]$. 
  \item\label{it:disjoint} Any ball $B(x,r)$ contains at most $A_1\delta^{-a_1}$ centres $x_i$ of pairwise disjoint balls $B(x_i,\delta r)$.
  \item\label{it:countable} Any disjoint family of balls in $X$ is at most countable.
  \item\label{it:separated} If $x,y\in X$ have $\rho (x,y)\geq r$, then the balls $B(x,r/(2A_0))$ and $B(y,r/(2A_0))$ are disjoint.
\end{enumerate}

A subset $\Omega\subseteq X$ is \emph{open} if for every $x\in\Omega$ there exist $\eps>0$ such that $B(x,\eps)\subseteq\Omega$. A subset $F\subseteq X$ is \emph{closed} if its complement is open. The usual proof of the fact that $F\subseteq X$ is closed, if and only if it contains its limit points, carries over to the quasi-metric spaces. However, some balls $B(x,r)$ may fail to be open. (E.g., consider $X=\{-1\}\cup[0,\infty)$ with the usual distance between all other pairs of points except $\rho(-1,0):=\frac12$. Then $B(-1,1)=\{-1,0\}$ does not contain any ball of the form $B(0,\eps)$, and hence cannot be open.)


\begin{theorem}\label{thm:cubes}
Suppose that constants $0<c_0\leq C_0 <\infty$ and $\delta \in (0,1)$ satisfy
\begin{equation}\label{parameter;conditions}
12A_0^3 C_0\delta \leq c_0.
\end{equation}
Given a set of points $\{z^k_{\alpha}\}_{\alpha}$, $\alpha\in \mathscr{A}_k$, for every $k\in\Z$, with the properties that
\begin{equation}\label{ehdot}
  \rho(z_{\alpha}^k,z_{\beta}^k)\geq c_0\delta^k\quad(\alpha\neq\beta),\qquad
  \min_{\alpha}\rho(x,z^k_{\alpha})<C_0\delta^k\quad\forall\,x\in X,
\end{equation}
we can construct families of sets $\tilde{Q}^k_{\alpha}\subseteq {Q}^k_{\alpha}\subseteq\bar{Q}^k_{\alpha}$ --- called open, half-open and closed \emph{dyadic cubes} --- such that:
\begin{equation}\label{eq:open-closed}
   \tilde{Q}^k_{\alpha}\text{ and }\bar{Q}^k_{\alpha}\text{ are the interior and closure of }Q^k_{\alpha};\\
\end{equation}
\begin{equation}\label{eq:nested}
   \text{if }\ell\geq k\text{, then either }Q^{\ell}_{\beta}\subseteq Q^k_{\alpha}\text{ or }Q^k_{\alpha}\cap Q^{\ell}_{\beta}=\varnothing;
\end{equation}
\begin{equation}\label{eq:cover}
  X=\bigcup_{\alpha}Q^k_{\alpha}\quad\text{(disjoint union)}\quad\forall k\in\Z;
\end{equation}
\begin{equation}\label{eq:contain}
  B(z^k_{\alpha},c_1\delta^k)\subseteq Q^k_{\alpha}\subseteq B(z^k_{\alpha},C_1\delta^k)=:B(Q^k_{\alpha})\text{, where $c_1:=(3A_0^2)^{-1}c_0$ and $C_1:=2A_0C_0$};
\end{equation}
\begin{equation}\label{eq:monotone}
   \text{if }k\leq\ell\text{ and }Q^{\ell}_{\beta}\subseteq Q^k_{\alpha}\text{, then }B(Q^{\ell}_{\beta})\subseteq B(Q^k_{\alpha}).
\end{equation}
The open and closed cubes $\tilde{Q}^k_{\alpha}$ and $\bar{Q}^k_{\alpha}$ depend only on the points $z^{\ell}_{\beta}$ for $\ell\geq k$. The half-open cubes $Q^k_{\alpha}$ depend on $z^{\ell}_{\beta}$ for $\ell\geq\min(k,k_0)$, where $k_0\in\Z$ is a preassigned number entering the construction.
\end{theorem}

To some extent, this combines the benefits of the alternative constructions of Christ and Sawyer--Wheeden: on the one hand, we obtain dyadic cubes on all length scales (rather than from a given level up), as in Christ's construction, and we also obtain an exact partition of the space (rather than up to measure zero) as in Sawyer and Wheeden. The fact that things would be slightly simpler if we started from a fixed finest level of partition, like Sawyer and Wheeden, is reflected by the dependence of the half-open cubes also on the points of the coarser scales once we go past the preassigned threshold level $k_0$.

The proof consists of several steps. 

\begin{lemma}[Partial order for dyadic points]\label{lem:partial;order}
Under the assumptions of Theorem~\ref{thm:cubes}, there is a partial order $\leq$ among the pairs $(k,\alpha)$ such that:
\begin{itemize}
  \item if $\rho(z^{k+1}_{\beta},z^k_{\alpha})<(2A_0)^{-1}c_0\delta^k$ then $(k+1,\beta)\leq(k,\alpha)$;
  \item if $(k+1,\beta)\leq(k,\alpha)$ then $\rho(z^{k+1}_{\beta},z^k_{\alpha})<C_0\delta^k$;
   \item for every $(k+1,\beta)$, there is exactly one $(k,\alpha)\geq(k+1,\beta)$, called its \emph{parent};
   \item for every $(k,\alpha)$, there are between $1$ and $M$ pairs $(k+1,\beta)\leq(k,\alpha)$, called its \emph{children};
  \item there holds $(\ell,\beta)\leq(k,\alpha)$ if and only if $\ell\geq k$ and there are $(j+1,\gamma_{j+1})\leq(j,\gamma_j)$ for all $j=k,k+1,\ldots,\ell-1$,
   for some $\gamma_k=\alpha,\gamma_{k+1},\ldots,\gamma_{\ell-1},\gamma_{\ell}=\beta$; then $(\ell,\beta)$ and $(k,\alpha)$ are called one another's \emph{descendant} and \emph{ancestor}, respectively.
\end{itemize}
\end{lemma}

\begin{proof}
Indeed, these properties essentially define $\leq$: Given a point $(k+1,\beta)$, check whether there exists $\alpha$ such that $\rho(z^k_{\alpha},z^{k+1}_{\beta})<(2A_0)^{-1}c_0\delta^k$. If one exists, it is necessarily unique by \eqref{ehdot} and we decree that $(k+1,\beta)\leq(k,\alpha)$. If no such good $\alpha$ exists, choose any $\alpha$ for which $\rho(z^k_{\alpha},z^{k+1}_{\beta})<C_0\delta^k$ (at least one such $\alpha$ exists by \eqref{ehdot}) and decree that $(k+1,\beta)\leq(k,\alpha)$. (From \eqref{ehdot} and the geometric doubling property, it follows readily that the index sets $\mathscr{A}_k$, $k\in \Z$, are countable. By assuming that the countable number of indices $\alpha$ are taken from $\N$, we could choose the smallest~$\alpha$, thereby eliminating the arbitrariness of this choice in the construction.) In either case, we decree that $(k+1,\beta)$ is not related to any other $(k,\gamma)$, and finally we extend $\leq$ by transitivity to obtain a partial ordering.

It only remains to check the claim concerning the number of children. If $(k+1,\beta)\leq (k,\alpha)$, we have $\rho(z_\alpha^{k},z_\beta^{k+1})<C_0\delta^k$, but also $\rho(z_\beta^{k+1},z^{k+1}_{\gamma})\geq c_0\delta^{k+1}$ for any $\gamma\neq\beta$. Thus the geometric doubling property implies that there can be at most boundedly many, say $M$, such $(k+1,\beta)$. Conversely, for $(k,\alpha)$, there exists at least one $(k+1,\beta)$ with $\rho(z_\alpha^{k},z_\beta^{k+1})<C_0\delta^{k+1}\leq(2A_0)^{-1}c_0\delta^k$, and thus $(k+1,\beta)\leq(k,\alpha)$, so there is at least one child.
\end{proof}

With the partial order defined, it is possible to formulate the rest of the construction, although proving all the stated properties needs some more work. As a preliminary version, we define for every $k\in \Z$ and every $\alpha$
\begin{equation*}
  \hat{Q}^k_{\alpha}:=\{z^{\ell}_{\beta}\colon (\ell,\beta)\leq(k,\alpha)\} .
\end{equation*}
Then
\begin{equation*}
  \bar{Q}^k_{\alpha}:=\overline{\hat{Q}^k_{\alpha}},
\end{equation*}
the closure of $\hat{Q}^k_{\alpha}$, and
\begin{equation*}
  \tilde{Q}^k_{\alpha}:=\Big(\bigcup_{\beta\neq\alpha}\bar{Q}^k_{\beta}\Big)^c.
\end{equation*}
It is clear from the definition that these depend only on $z^{\ell}_{\beta}$ for $\ell\geq k$.

In the following section we prove:

\begin{proposition}[Properties of closed and open dyadic cubes]\label{prop:cubes}
Suppose that for every $k\in \Z$ we have a set of points with properties \eqref{ehdot} and constants that satisfy \eqref{parameter;conditions} as in Theorem~\ref{thm:cubes}. Then the cubes $\tilde{Q}^k_{\alpha}$ and $\bar{Q}^k_{\alpha}$ satisfy:
\begin{equation}\label{eq:interior-closure}
  \tilde{Q}^k_{\alpha}\text{ and }\bar{Q}^k_{\alpha}\text{ are one another's interior and closure, respectively};
\end{equation}
\begin{equation}\label{eq:nested-prelim}
  \text{for }\ell\geq k\text{, there holds }\tilde{Q}^{\ell}_{\beta}\subseteq\tilde{Q}^k_{\alpha}\text{ if $(\ell,\beta)\leq(k,\alpha)$, and }
  \tilde{Q}^{\ell}_{\beta}\cap\bar{Q}^k_{\alpha}=\bar{Q}^{\ell}_{\beta}\cap\tilde{Q}^k_{\alpha}=\varnothing\text{ else};
\end{equation}
\begin{equation}\label{eq:closed-cover}
  X=\bigcup_{\alpha}\bar{Q}^k_{\alpha}\quad\text{(possibly with overlap)}\quad\forall k\in\Z;
\end{equation}
\begin{equation}\label{eq:contain-prelim}
  (i)\quad B(z^k_{\alpha},c_1\delta^k)\subseteq\tilde{Q}^k_{\alpha},\qquad (ii)\quad\bar{Q}^k_{\alpha}\subseteq B(z^k_{\alpha},C_1\delta^k);
\end{equation}
\begin{equation}\label{eq:monot-prelim}
  \text{if }(\ell,\beta)\leq(k,\alpha)\text{ then }B(z^{\ell}_{\beta},C_1\delta^{\ell})\subseteq B(z^{k}_{\alpha},C_1\delta^k).
\end{equation}
Moreover,
\begin{equation}\label{eq:unionOfClosed}
     \bar{Q}^k_{\alpha}=\bigcup_{\beta:(\ell,\beta)\leq(k,\alpha)}\bar{Q}^{\ell}_{\beta}\qquad\forall k\leq\ell. 
\end{equation}
\end{proposition}

Assuming this result, which contains the essence of Theorem~\ref{thm:cubes}, we can complete the proof of the Theorem by the following lemma: 

\begin{lemma}[Construction of half-open cubes]\label{lem:finalise}
Assuming Proposition~\ref{prop:cubes}, we can construct $Q^k_{\alpha}$ which satisfy the assertions of Theorem~\ref{thm:cubes}.
\end{lemma}

\begin{proof}
Here it is convenient to assume that the pairs $(k,\alpha)$ are parameterized by $\alpha\in\N$ for each $k\in\Z$.
For the given threshold level $k=k_0$, we define recursively
\begin{equation*}
  {Q}^{k_0}_0:=\bar{Q}^{k_0}_0,\qquad
  {Q}^{k_0}_{\alpha}:=\bar{Q}^{k_0}_{\alpha}\setminus\bigcup_{\beta=0}^{\alpha-1}{Q}^{k_0}_{\beta},\quad\alpha\geq 1.
\end{equation*}
By construction, it is clear that the ${Q}^{k_0}_{\alpha}$ are pairwise disjoint and satisfy
\begin{equation}\label{eq:between}
  \bar{Q}^{k_0}_{\alpha}\supseteq Q^{k_0}_{\alpha}\supseteq\bar{Q}^{k_0}_{\alpha}\setminus\bigcup_{\beta\neq\alpha}\bar{Q}_{\beta}^{k_0}
  \supseteq\tilde{Q}^{k_0}_{\alpha}\setminus(\tilde{Q}^{k_0}_{\alpha})^c=\tilde{Q}^{k_0}_{\alpha},\qquad
  \bigcup_{\beta=0}^{\alpha}{Q}^{k_0}_{\beta}=\bigcup_{\beta=0}^{\alpha}\bar{Q}^{k_0}_{\beta}\underset{\alpha\to\infty}{\longrightarrow} X.
\end{equation}

For $k<k_0$ we define
\begin{equation*}
  {Q}^k_{\alpha}:=\bigcup_{\beta:(k_0,\beta)\leq(k,\alpha)}{Q}^{k_0}_{\beta}.
\end{equation*}
Then clearly $Q^k_{\alpha}\subseteq\bar{Q}^k_{\alpha}$, these partition $X$ for a fixed $k$, and \eqref{eq:nested} holds for all $k\leq\ell\leq k_0$. Finally,
\begin{equation*}
  (Q^k_{\alpha})^c=\bigcup_{\beta:(k_0,\beta)\not\leq(k,\alpha)}{Q}^{k_0}_{\beta}
  \subseteq\bigcup_{\gamma\neq\alpha}\bigcup_{\beta:(k_0,\beta)\not\leq(k,\gamma)}\bar{Q}^{k_0}_{\beta}
  =\bigcup_{\gamma\neq\alpha}\bar{Q}^k_{\gamma}=(\tilde{Q}^k_{\alpha})^c,
\end{equation*}
so that $\tilde{Q}^k_{\alpha}\subseteq Q^k_{\alpha}$.

For $k>k_0$, we proceed by induction as follows. Suppose that the cubes ${Q}^{\ell}_{\alpha}$, $\ell\leq k-1$, are already defined as required.  For every $\alpha$, consider the finitely many $(k,\beta)\leq(k-1,\alpha)$, and relabel them temporarily with $\beta=0,1,\ldots$ (up to some finite number). Then define
\begin{equation*}
  {Q}^k_0:={Q}^{k-1}_{\alpha}\cap\bar{Q}^k_0,\qquad
  {Q}^k_{\beta}:={Q}^{k-1}_{\alpha}\cap\bar{Q}^k_{\beta}\setminus\bigcup_{\gamma=0}^{\beta-1}{Q}^k_{\gamma},\quad
  \beta\geq 1.
\end{equation*} 
Disjointness is clear, and as in \eqref{eq:between} we get $\bar{Q}^k_{\beta}\supseteq Q^k_{\beta}\supseteq Q^{k-1}_{\alpha}\cap\tilde{Q}^k_{\beta}\supseteq \tilde{Q}^{k-1}_{\alpha}\cap\tilde{Q}^k_{\beta}=\tilde{Q}^k_{\beta}$, and
\begin{equation*}
  \bigcup_{\beta:(k,\beta)\leq(k-1,\alpha)}{Q}^k_{\beta}
  ={Q}^{k-1}_{\alpha}\cap\bigcup_{\beta:(k,\beta)\leq(k-1,\alpha)}\bar{Q}^k_{\beta}={Q}^{k-1}_{\alpha}.
\end{equation*}
This easily implies \eqref{eq:nested} and \eqref{eq:cover} in all the remaining cases.

Finally, \eqref{eq:interior-closure}$\Rightarrow$\eqref{eq:open-closed} and \eqref{eq:contain-prelim}$\Rightarrow$\eqref{eq:contain} are clear from $\tilde{Q}^k_{\alpha}\subseteq Q^k_{\alpha}\subseteq \bar{Q}^k_{\alpha}$. To see that \eqref{eq:monot-prelim}$\Rightarrow$\eqref{eq:monotone}, observe that if $k\leq\ell$ and $Q^{\ell}_{\beta}\subseteq Q^k_{\alpha}$, then \eqref{eq:nested} and \eqref{eq:nested-prelim} imply that $(\ell,\beta)\leq(k,\alpha)$, and thus \eqref{eq:monot-prelim} can be used.
\end{proof}

\begin{remark}
We work with several adjacent sets of cubes, as stated. Given a system of dyadic points, the preliminary cubes $\hat{Q}_{\alpha}^{k}$ determine the open cubes $\tilde{Q}_{\alpha}^{k}$ and the closed cubes $\bar{Q}_{\alpha}^{k}$ unambiguously but not the half-open cubes $Q_{\alpha}^{k}$ as their construction involves selection. If we know the final half-open cubes $Q_{\alpha}^{k}$ as point sets, they determine the open cubes $\tilde{Q}_{\alpha}^{k}$ and the closed cubes $\bar{Q}_{\alpha}^{k}$ unambiguously as these are nothing but the set of interior points and the closure respectively. The family of the preliminary cubes $\hat{Q}_{\alpha}^{k}$, however, is not uniquely determined by the point sets $Q_{\alpha}^{k}$. But usually we think that the cubes $Q_{\alpha}^{k}$ carry the information of their centre point $z^k_\alpha$ and generation $k$ as well. The ideology here is the very same as when defining a metric ball.
\end{remark}

\subsection{Existence of dyadic points}\label{existence;dyadicpoints}
Consider a maximal collection of points $x_\alpha^{k}\in X$ satisfying the two inequalities
\begin{equation}
  \rho(x_{\alpha}^k,x_{\beta}^k)\geq c_0\delta^k\quad(\alpha\neq\beta),\qquad
  \min_{\alpha}\rho(x,x^k_{\alpha})<C_0\delta^k\quad\forall\,x\in X
\end{equation}
with constants $c_0=1=C_0$. It follows from the maximality argument that such a point set exists for any given $\delta\in (0,1)$ and every $k\in \Z$. From the first condition and the geometric doubling property it follows that a minimum in the second condition is indeed attained. Note that we may, of course, choose maximal point sets in such a way that given a fixed point $x_0\in X$, for every $k\in \Z$, there exists $\alpha$ such that $x^k_\alpha = x_0$. Finally, we may choose $\delta$ such that the restriction (\ref{parameter;conditions}) introduced in Theorem~\ref{thm:cubes} holds.

%
%

\section{Verification of the properties}

This section contains the proof of Proposition~\ref{prop:cubes}, which consists of the more technical aspects of Theorem~\ref{thm:cubes}. Suppose that for every $k\in \Z$ we have a set of points with properties (\ref{ehdot}) and constants that satisfy (\ref{parameter;conditions}) as in Theorem~\ref{thm:cubes}. Recall that $c_1:=(3A_0^2)^{-1}c_0$ and $C_1:=2A_0C_0$. We start with simple inclusion properties.

\begin{lemma}\label{lem:descendant}
If $(\ell,\beta)\leq (k,\alpha)$, then $\rho (z_\alpha^{k},z_\beta^{\ell})< C_1\delta^{k}$.
\end{lemma}

\begin{proof}
Consider the chain 
\begin{equation*}
  (k,\alpha)=(k, \gamma_0)\geq(k+1, \gamma_1)\geq \ldots\geq (k+(\ell-k), \gamma_{\ell-k})=(\ell,\beta)
\end{equation*}
with $ \rho (z_{\gamma_i}^{k+i},z_{\gamma_{i+1}}^{k+i+1})\leq C_0\delta^{k+i}$ for all $i\in \lbrace 0, \ldots , l-k-1 \rbrace$.
By iterating the triangle inequality,
\begin{equation*}
\begin{split}
  \rho (z_\alpha^{k},z_\beta^{\ell})
  & \leq \sum_{i=0}^{k-\ell-1} A_0^{i+1}\rho(z_{\gamma_i}^{k+i},z_{\gamma_{i+1}}^{k+i+1}) \\
  & \leq \sum_{i=0}^{k-\ell-1} A_0^{i+1}C_0\delta^{k+i}<\frac{A_0 C_0\delta^k}{1-A_0\delta}\leq 2A_0C_0\delta^k.\qedhere
\end{split}
\end{equation*}
\end{proof}

\begin{lemma}[Containing balls; \eqref{eq:contain-prelim}(ii) and \eqref{eq:monot-prelim}]
We have $\bar{Q}^k_{\alpha}\subseteq B(z^k_{\alpha},C_1\delta^k)$, and also $B(z^{\ell}_{\beta},C_1\delta^k)\subseteq B(z^k_{\alpha},C_1\delta^k)$ for all $(\ell,\beta)\leq(k,\alpha)$.
\end{lemma}

\begin{proof}
For the first inclusion, let $x\in\bar{Q}^k_{\alpha}$; hence it is a limit of some $x_r\in\hat{Q}^k_{\alpha}$, $r\in\Z_+$. If $x_r=z^k_{\alpha}$ infinitely often, then also $x=z^k_{\alpha}$, and there is nothing to prove. Otherwise, infinitely many $x_r$ are of the form $z^\ell_\beta$ for some $(\ell,\beta)=(\ell(r),\beta(r))\leq(k+1,\gamma)=(k+1,\gamma(r))\leq(k,\alpha)$, and then
\begin{equation}\label{eq:compu}
\begin{split}
  \rho(z^k_{\alpha},x)
  &\leq A_0\rho(z^k_{\alpha},z^{k+1}_{\gamma})+A_0^2\rho(z^{k+1}_{\gamma},z^{\ell}_{\beta})+
  A_0^2\rho(z^{\ell}_{\beta},x) \\
  &\leq A_0 C_0\delta^k + A_0^2\cdot 2A_0C_0\delta^{k+1}+A_0^2 \rho(x_r,x) \\
  &< A_0 C_0\delta^k+\tfrac12A_0C_0\delta^{k}+\tfrac12A_0 C_0\delta^{k}
  =2A_0C_0\delta^k,
\end{split}
\end{equation}
for such $x_r$ with $r\geq r_0$, since $\ell\geq k+1$ and $4A_0^2\delta\leq 1$.

The inclusion between the balls is clear if $\ell=k$, so let again $(\ell,\beta)\leq(k+1,\gamma)\leq(k,\alpha)$. Let $x\in B(z^{\ell}_{\beta},C_1\delta^{\ell})$. As in \eqref{eq:compu}, we deduce (now using $\rho(z^{\ell}_{\beta},x)\leq C_1\delta^{\ell}$ instead of $\rho(z^{\ell}_{\beta},x)=\rho(z^{\ell}_{\beta},x_r)$) that
\begin{equation*}
  \rho(z^k_{\alpha},x)
  < A_0 C_0\delta^k+A_0^2\cdot 2A_0C_0\delta^{k+1}+A_0^2 \cdot C_1\delta^{\ell}\leq 2A_0C_0\delta^k.\qedhere
\end{equation*}
\end{proof}

\begin{lemma}\label{closed:union}
For any $\Lambda$, $\bigcup_{\alpha\in\Lambda}\bar{Q}^k_{\alpha}$ is the closure of $\bigcup_{\alpha\in\Lambda}\hat{Q}^k_{\alpha}$; in particular, this union is closed.
Hence the open cubes $\displaystyle \tilde{Q}^k_{\alpha}:=\Big(\bigcup_{\gamma\neq\alpha}\bar{Q}^k_{\gamma}\Big)^c$ are indeed open sets.
\end{lemma}

\begin{proof}
Let $k\in\Z$ be fixed. It is clear that each $\bar{Q}^k_{\alpha}$, $\alpha\in\Lambda$, is a subset of the closure of $\bigcup_{\alpha\in\Lambda}\hat{Q}^k_{\alpha}$.
From the geometric doubling property and the inclusion $\hat{Q}^k_{\alpha}\subseteq B(z^k_{\alpha},C_1\delta^k)$, it follows readily that a bounded set can intersect at most finitely many different $\hat{Q}^k_{\alpha}$. Hence, if a convergent, thus bounded, sequence of points $x_r$ belong to $\bigcup_{\alpha\in\Lambda}\bar{Q}^k_{\alpha}$, then they belong to some sub-union with a finite $\Lambda_1\subseteq\Lambda$ in place of $\Lambda$. A union of finitely many closed sets is closed, so also the limit of the sequence $(x_r)$ must belong to the same union. Thus all limit points of $\bigcup_{\alpha\in\Lambda}\hat{Q}^k_{\alpha}$ belong to $\bigcup_{\alpha\in\Lambda}\bar{Q}^k_{\alpha}$.
\end{proof}

\begin{lemma}[Unions of closed cubes; \eqref{eq:closed-cover} and \eqref{eq:unionOfClosed}]\label{lemma:closedcovering}
For all $k,\ell\in\Z$ with $\ell>k$, we have
\begin{equation*}
  X=\bigcup_{\alpha}\bar{Q}^k_{\alpha},\qquad\bar{Q}^k_{\alpha}=\bigcup_{\beta:(\ell,\beta)\leq(k,\alpha)}\bar{Q}^{\ell}_{\beta}.
\end{equation*}
\end{lemma}

\begin{proof}
The union $\bigcup_{\alpha}\hat{Q}^k_{\alpha}$ contains the points $z^{\ell}_{\beta}$ with $\ell\geq k$ and $\beta$ arbitrary, which are dense in $X$ by \eqref{ehdot}. Hence the closure of this union is $X$, but it is also equal to $\bigcup_{\alpha}\bar{Q}^k_{\alpha}$ by Lemma~\ref{closed:union}.

We turn to the second identity, first with $\ell=k+1$. It is clear that
\begin{equation*}
  \hat{Q}^k_{\alpha}=\{z^k_{\alpha}\}\cup\bigcup_{\beta:(k+1,\beta)\leq(k,\alpha)}\hat{Q}^{k+1}_{\beta}; 
\end{equation*}
hence by taking closures with the help of Lemma~\ref{closed:union},
\begin{equation*}
  \bar{Q}^k_{\alpha}=\{z^k_{\alpha}\}\cup\bigcup_{\beta:(k+1,\beta)\leq(k,\alpha)}\bar{Q}^{k+1}_{\beta}.
\end{equation*}
Since the cubes $\bar{Q}^{k+1}_{\beta}$ cover all $X$, it is clear that $z^k_{\alpha}\in\bar{Q}^{k+1}_{\beta}\subseteq B(z^{k+1}_{\beta},C_1\delta^{k+1})$ for some $\beta$, and we only need to check that $(k+1,\beta)\leq(k,\alpha)$. But this follows, using $4A_0^2 C_0\delta\leq c_0$, from
\begin{equation*}
  \rho(z^k_{\alpha},z^{k+1}_{\beta})<C_1\delta^{k+1}=2A_0 C_0\delta\cdot\delta^k\leq\frac{c_0}{2A_0}\delta^k.
\end{equation*}
The case of a general $\ell>k$ follows by an $(\ell-k)$-fold iteration of  the identity for $\ell=k+1$.
\end{proof}

\begin{lemma}\label{lem:ballInside}
$\bar{Q}^k_{\beta}\subseteq B(z^k_{\alpha},c_1\delta^k)^c$ for $\beta\neq\alpha$.
\end{lemma}

\begin{proof}
By Lemma~\ref{lemma:closedcovering}, we need to show that $\bar{Q}^{k+1}_{\gamma}\subseteq B(z^k_{\alpha},c_1\delta^k)^c$ for all $(k+1,\gamma)\leq(k,\beta)$. If not, then $\bar{Q}^{k+1}_{\gamma}\subseteq B(z^{k+1}_{\gamma},C_1\delta^{k+1})$ and $B(z^k_{\alpha},c_1\delta^k)$ have a common point $x$; whence
\begin{equation*}
\begin{split}
  \rho(z^{k+1}_{\gamma},z^k_{\alpha})
  &\leq A_0\rho(z^{k+1}_{\gamma},x)+A_0\rho(x,z^k_{\alpha})
  <A_0 C_1\delta^{k+1}+A_0 c_1\delta^k \\
  &=(2A_0^2 C_0\delta+\frac{c_0}{3A_0})\delta^k
  \leq(\frac{c_0}{6A_0}+\frac{c_0}{3A_0})\delta^k=\frac{c_0}{2A_0}\delta^k,
\end{split}
\end{equation*}
by $12A_0^3 C_0\delta\leq c_0$, and this implies that $(k+1,\gamma)\leq(k,\alpha)$, a contradiction with $\alpha\neq\beta$.
\end{proof}

\begin{lemma}[Nestedness and contained balls; \eqref{eq:nested-prelim} and  \eqref{eq:contain-prelim}(i)]
If $\ell\geq k$, then $\tilde{Q}^{\ell}_{\beta}\subseteq\tilde{Q}^k_{\alpha}$ for $(\ell,\beta)\leq(k,\alpha)$, or
$\tilde{Q}^{\ell}_{\beta}\cap\bar{Q}^k_{\alpha}=\bar{Q}^{\ell}_{\beta}\cap\tilde{Q}^k_{\alpha}=\varnothing$ otherwise.
Moreover, $\hat{Q}^k_{\alpha}\subseteq\tilde{Q}^k_{\alpha}\subseteq\bar{Q}^k_{\alpha}$ and $B(z^k_{\alpha},c_1\delta^k)\subseteq\tilde{Q}^k_{\alpha}$.
\end{lemma}

\begin{proof}
Let first $(\ell,\beta)\leq(k,\alpha)$. Then
\begin{equation*}
  (\tilde{Q}^k_{\alpha})^c
  =\bigcup_{\gamma\neq\alpha}\bar{Q}^k_{\gamma}
  =\bigcup_{\gamma\neq\alpha}\bigcup_{\eta:(\ell,\eta)\leq(k,\gamma)}\bar{Q}^{\ell}_{\eta}
  =\bigcup_{\eta:(\ell,\eta)\not\leq(k,\alpha)}\bar{Q}^{\ell}_{\eta}
  \subseteq\bigcup_{\eta\neq\beta}\bar{Q}^{\ell}_{\eta}
  =(\tilde{Q}^{\ell}_{\beta})^c;
\end{equation*}
hence $\tilde{Q}^{\ell}_{\beta}\subseteq\tilde{Q}^k_{\alpha}$. By Lemma~\ref{lem:ballInside}, we have
\begin{equation*}
  B(z^{\ell}_{\beta},c_1\delta^{\ell})^c\supseteq\bigcup_{\eta\neq\beta}\bar{Q}^{\ell}_{\eta}=(\tilde{Q}^{\ell}_{\beta})^c;
\end{equation*}
thus $z^{\ell}_{\beta}\in B(z^{\ell}_{\beta},c_1\delta^{\ell})\subseteq\tilde{Q}^{\ell}_{\beta}\subseteq\tilde{Q}^k_{\alpha}$. This gives both $\hat{Q}^k_{\alpha}\subseteq\tilde{Q}^k_{\alpha}$ and $B(z^k_{\alpha},c_1\delta^k)\subseteq\tilde{Q}^k_{\alpha}$.

To see that $\tilde{Q}^k_{\alpha}\subseteq\bar{Q}^k_{\alpha}$, observe from Lemma~\ref{lemma:closedcovering} that $X=\bigcup_{\alpha}\bar{Q}^k_{\alpha}=\bar{Q}^k_{\alpha}\cup\Big(\bigcup_{\beta\neq\alpha}\bar{Q}^k_{\beta}\Big)$, so that
\begin{equation*}
  \tilde{Q}^k_{\alpha}=\Big(\bigcup_{\beta\neq\alpha}\bar{Q}^k_{\beta}\Big)^c\subseteq\bar{Q}^k_{\alpha}.
\end{equation*}

Let then $(\ell,\beta)\not\leq(k,\alpha)$, and thus $(\ell,\beta)\leq(k,\gamma)$ for some $\gamma\neq\alpha$. By what was already proven, $\tilde{Q}^{\ell}_{\beta}\subseteq\tilde{Q}^k_{\gamma}\subseteq(\bar{Q}^k_{\alpha})^c$, and (taking closures) $\bar{Q}^{\ell}_{\beta}\subseteq\bar{Q}^k_{\gamma}\subseteq\bigcup_{\eta\neq\alpha}\bar{Q}^k_{\eta}=(\tilde{Q}^k_{\alpha})^c$.
\end{proof}

\begin{lemma}[Closure and interior; \eqref{eq:interior-closure}]\label{lem:closure_interior}
The cubes $\bar{Q}^k_{\alpha}$ and $\tilde{Q}^k_{\alpha}$ are each other's closure and interior.
\end{lemma}

\begin{proof}
From $\hat{Q}^k_{\alpha}\subseteq\tilde{Q}^k_{\alpha}\subseteq\bar{Q}^k_{\alpha}$ and the fact that $\bar{Q}^k_{\alpha}$ is the closure of $\hat{Q}^k_{\alpha}$, it is clear that it is also the closure of $\tilde{Q}^k_{\alpha}$.

Concerning the interior, it is clear that the open set $\tilde{Q}^k_{\alpha}\subseteq\bar{Q}^k_{\alpha}$ is a subset of the interior of $\bar{Q}^k_{\alpha}$. For the other direction, observe that $\hat{Q}^k_{\beta}\subseteq\tilde{Q}^k_{\beta}\subseteq(\bar{Q}^k_{\alpha})^c$ for all $\beta\neq\alpha$; hence
\begin{equation*}
  \overline{(\bar{Q}^k_{\alpha})^c}\supseteq\overline{\bigcup_{\beta\neq\alpha}\hat{Q}^k_{\beta}}
  =\bigcup_{\beta\neq\alpha}\bar{Q}^k_{\beta}=(\tilde{Q}^k_{\alpha})^c.
\end{equation*}
Thus the interior of $\bar{Q}^k_{\alpha}$, which is the complement of $\overline{(\bar{Q}^k_{\alpha})^c}$, is a subset of $(\tilde{Q}^k_{\alpha})^{cc}=\tilde{Q}^k_{\alpha}$.
\end{proof}

\section{Adjacent dyadic systems}\label{sec:adjacent}
In this section we will prove the following theorem.

\begin{theorem}\label{thm:adjacent;systems}
Given a set of reference points $\{x^k_\alpha \},k\in \Z, \alpha\in\mathscr{A}_k$, suppose that constant $\delta\in (0,1)$ satisfies $96A_0^6\delta\leq 1$. Then there exists a finite collection of families $\mathscr{D}^t$, $t = 1, 2, \ldots , K=K(A_0, A_1, \delta) <\infty$, where each $\mathscr{D}^t$ is a collection of dyadic cubes, associated to dyadic points $\{z^k_\alpha \},k\in \Z, \alpha\in\mathscr{A}_k$, with the properties \eqref{eq:open-closed}--\eqref{eq:monotone} of Theorem~\ref{thm:cubes}. In addition, the following property is satisfied: 
\begin{equation}\label{property:ball;included}
\hspace*{-0.145cm}\text{for every ball $B=B(x,r)\subseteq X$, there exists $t$ and $Q\in \mathscr{D}^t$ with }
B\subseteq Q\text{ and } \diam (Q)\leq Cr.
\end{equation}
The constant $C<\infty$ in \eqref{property:ball;included} only depends on the quasi-metric constant $A_0$ and parameter $\delta$.
\end{theorem}

We will also prove the following variant of Theorem~\ref{thm:adjacent;systems}:

\begin{proposition}\label{prop:adjacent;systems}
Given a fixed point $x_0\in X$, there exists a finite collection of families $\mathscr{D}^t$, $t = 1, 2, \ldots , K=K(A_0, A_1, \delta) <\infty$, where each $\mathscr{D}^t$ is a collection of dyadic cubes with the properties \eqref{eq:open-closed}--\eqref{eq:monotone} of Theorem~\ref{thm:cubes}, and the property (\ref{property:ball;included}) of Theorem~\ref{thm:adjacent;systems} is satisfied. In addition, the following property is satisfied: For every $t = 1, 2, \ldots , K$, for every $k$, there exists $\alpha$ such that $x_0=z^k_\alpha$, the center point of $Q^k_\alpha\in \mathscr{D}^t$. 
\end{proposition}

\subsection{Reference dyadic points} 
Recall from \ref{existence;dyadicpoints} that for every $k\in \Z$ there exists a point set $\{ x^k_\alpha\}_{\alpha\in\mathscr{A}_k}$ such that
\[\rho(x_{\alpha}^k,x_{\beta}^k)\geq \delta^k\quad(\alpha\neq\beta),\qquad
  \min_{\alpha}\rho(x,x^k_{\alpha})<\delta^k\quad\forall\,x\in X. \]
We will refer to the set $\{x^k_\alpha \}_{k,\alpha}$ of dyadic points as the set of \textit{reference points}.

Suppose that $\delta\in (0,1)$ satisfies $96A_0^6\delta\leq 1$, and set $c_0:=(4A_0^2)^{-1}$. In particular, $\delta<c_0$ and 
\begin{equation}\label{def;b0}
\frac{1}{2A_0^2}-\delta >\frac{1}{2A_0^2}-c_0=c_0. 
\end{equation}

\begin{definition}\label{konflikti} Reference points $x_{\alpha}^k$ and $x_{\beta}^k$, $\alpha\neq\beta$, of same generation are \textit{in conflict}  if
\[\rho(x_{\alpha}^k,x_{\beta}^k,)< c_0\delta^{k-1}.\]
\end{definition}

\begin{definition} Reference points $x_{\alpha}^k$ and $x_{\beta}^k$, $\alpha\neq\beta$, of same generation are \textit{neighbours} if there occurs a conflict between their children. More precisely, the points  $x_{\alpha}^k$ and $x_{\beta}^k$, $\alpha\neq\beta$, are neighbours  if there exist points $(k+1,\gamma)\leq(k,\alpha)$ and $(k+1, \sigma)\leq(k,\beta)$ such that
$\rho(x_{\gamma}^{k+1},x_{\sigma}^{k+1})< c_0\delta^{k}$.
\end{definition}

Recall from \ref{lem:partial;order} that due to the doubling property, a dyadic point can have at most $M$ children (a constant independent of the point). By similar arguments, also the number of neighbours of a dyadic point is bounded from above by a fixed constant, say $L$.

\subsection{Labeling of the reference points}\label{definition:labels}
Fix $k\in \Z$. We label the reference points $x_{\alpha}^{k}$ of generation $k$ as follows: Begin with some index pair $(k,\alpha)$ and label it with number $0$. Then, for every $(k,\beta)$, $\beta\neq \alpha$, check whether any of its neighbours (boundedly many) already have a label. If not, label it with number $0$. Otherwise, pick the \textit{smallest} positive integer not yet in use among the neighbours. As the number of neighbours a point can have is bounded above by constant $L$, every point $x^k_\alpha$ gets a \textit{primary label} $\lab_1(k,\alpha):=\ell$ not bigger than $L$. Furthermore, we have the following: if $(k,\alpha)$ and $(k,\beta)$, $\alpha\neq\beta$, have the same label $\ell\in \{0, \ldots ,L\}$, they are not neighbours.

Next we label the reference points $x_{\gamma}^{k+1}$ of the following generation $k+1$ with duplex labels: If $\lab_1(k,\alpha)=\ell$, each of its children $(k+1,\beta)\leq (k,\alpha)$ (boundedly many) gets a different \textit{duplex label} $\lab_2(k+1,\beta):=(\ell,m), m=m(\beta)\in\{1,2,\ldots ,M \}$. Then we have the following: if $(k+1,\gamma)$ and $(k+1,\sigma)$, $\gamma\neq\sigma$, have the same primary label $\ell\in \{0, \ldots ,L\}$, they are not in conflict.

We next define new dyadic points $z^k_\alpha$ of generation $k$ by selecting them from the set of reference points of generation $k+1$. We will first allow this selection with only little restriction and then consider a more specific choice.

\begin{definition}[General selection rule]\label{general:rule}
For every $k\in\Z$, pick $\ell=\ell_k\in \{ 0, \ldots ,L\}$ which we refer to as the master label. For every $\alpha$, check if $\lab_1(k,\alpha)=\ell$. If so, pick any $(k+1,\beta)\leq (k,\alpha)$ and declare that $(k,\alpha)\searrow (k+1,\beta)$ and $(k,\alpha)\not\searrow (k+1,\gamma)$ for every  $\gamma\neq \beta$. Also decree that $z^k_\alpha:=x^{k+1}_\beta$. 

Otherwise, pick some $(k+1,\beta)\leq (k,\alpha)$ with $\rho(x^{k+1}_\beta,x^k_\alpha)<\delta^{k+1}$ (such a child always exists) and declare that $(k,\alpha)\searrow (k+1,\beta)$ and $(k,\alpha)\not\searrow (k+1,\gamma)$ for every  $\gamma\neq \beta$. Also decree that $z^k_\alpha:=x^{k+1}_\beta$. Note that, by construction, every $(k,\alpha)$ is related to some $(k+1,\beta)$ by the relation $\searrow$ but it may or may not be related to any $(k-1,\sigma)$. 
\end{definition}

The point sets obtained from the reference points by the general selection rule have the distribution property \eqref{ehdot} of Theorem~\ref{thm:cubes}:

\begin{lemma}\label{dyadic_points} 
Set $c_0:=(4A_0^2)^{-1}$ and $C_0:=2A_0$. Let $\{z^k_\alpha\}$ be the set of new dyadic points obtained by the general selection rule from the reference points. Then, for every $k\in \Z$ we have
\begin{equation*}\label{distribution;points}
\rho (z_\alpha^k, z_\beta^k)\geq c_0\delta^k, \, \alpha\neq \beta , 
\end{equation*}
and for every $x\in X$ and every $k\in \Z$ we find $\alpha$ such that
\begin{equation*}\label{distribution;point}
\rho (x,z_\alpha^k)< C_0\delta^k . 
\end{equation*}
\end{lemma}

\begin{proof}
Let us fix $k\in \Z$ and $\ell=\ell_k\in \{ 0, \ldots ,L\}$. By the general selection rule, $z_\alpha^k=x_{\gamma}^{k+1}$ and $z_\beta^k=x_{\sigma}^{k+1}$ for some reference points $x_{\gamma}^{k+1}$ and $x_{\sigma}^{k+1}$. First assume that at least one of the reference points $x_{\alpha}^{k}$ and $x_{\beta}^{k}$ has primary label different from the master label $\ell$. We may, without loss of generality, assume that $\lab_1(k,\alpha)\neq \ell$. This implies $\rho(z_{\alpha}^{k},x_{\alpha}^k)<\delta^{k+1}$. Since $(k+1,\sigma)$ is not a child of $(k,\alpha)$, we have $\rho(z_{\beta}^{k},x_{\alpha}^k)\geq (2A_0)^{-1}\delta^{k}$. Thus,
\begin{equation*}
(2A_0)^{-1}\delta^{k}\leq \rho(z_{\beta}^{k},x_{\alpha}^k)\leq A_0\rho(x_{\alpha}^k,z_{\alpha}^{k})+A_0\rho(z_{\alpha}^{k},z_{\beta}^k)\leq A_0\delta^{k+1}+A_0\rho(z_{\alpha}^{k},z_{\beta}^k),
\end{equation*} 
or equivalently, by (\ref{def;b0}),
\begin{equation*}
\rho(z_{\alpha}^{k},z_{\beta}^k)\geq \frac{1}{A_0}\left( \frac{\delta^k}{2A_0}-A_0\delta^{k+1}\right) =\left(\frac{1}{2A_0^2}-\delta \right)\delta^k >c_0\delta^k .
\end{equation*}
Otherwise, both $x_{\alpha}^{k}$ and $x_{\beta}^{k}$ have primary label $\ell$. In this case, the points $x_{\alpha}^k$ and $x_{\beta}^k$ are not neighbours and there is no conflict between their children. The first assertion follows.

Fix $x\in X$. There exists $\alpha$ such that $\rho (x,x_{\alpha}^k)<\delta^k$ where $x_{\alpha}^k$ is a reference point. Point $z_{\alpha}^k$ is a child of $x_{\alpha}^k$ thus, $\rho(x_{\alpha}^k,z_{\alpha}^k)<\delta^k$. It follows that
\[\rho (x,z_{\alpha}^k)\leq A_0\rho (x,x_{\alpha}^k)+A_0\rho (x_{\alpha}^k,z_{\alpha}^k)< A_0\delta^k+A_0\delta^k = 2A_0\delta^k . \]
\end{proof}


\begin{definition}[Specific selection rule]\label{spesific;rule}
Fix $(\ell,m)\in \{0, \ldots ,L\}\times \{1, \ldots ,M\}$. For every index pair $(k,\alpha)$, check whether there exists $(k+1,\beta)\leq (k,\alpha)$ with label pair $(\ell,m)$. If so, decree that $z^k_\alpha:=x_\beta^{k+1}$. Otherwise, pick some $(k+1,\beta)\leq (k,\alpha)$ with $\rho(x_{\beta}^{k+1},x_{\alpha}^k)<\delta^{k+1}$ and decree $z_\alpha^k:=x_\beta^{k+1}$. 

Note that the specific selection rule is more precise than the general selection rule. Indeed, in case a reference point has primary label same as the master label $\ell$, we do not just choose any child but the one with duplex label $(\ell,m)$ (if one exists). Thus, it is a special case of the general selection rule. In particular, the point sets obtained from the reference points by the specific selection rule satisfy the distribution properties of Lemma \ref{dyadic_points}. 

Let $\varphi$ be a bijection $\{0, \ldots ,L\}\times \{1, \ldots ,M\}\to \{1, \ldots ,K\}\subset \N$,  $(\ell,m)\mapsto t$. We identify $t=\varphi(\ell ,m)$ with $(\ell,m)$. 
Each $t$ gives rise to a set $\{^t\!z_\alpha^k : k\in \Z, \alpha\in \mathscr{A}_k\}$ of new dyadic points associated with the duplex label $(\ell,m)=t$. Note that, by repeating the specific selection rule for every ordered pair of labels $(\ell,m)$, Lemma~\ref{dyadic_points} and Theorem~\ref{thm:cubes} complete the proof of the first part of Theorem~\ref{thm:adjacent;systems}. We denote by $\mathscr{D}^t$ the family of dyadic cubes $Q^k_\alpha ={} ^t\! Q^k_\alpha$ corresponding to the point set $\{^t\!z_\alpha^k : k\in \Z, \alpha\in \mathscr{A}_k\}$, $t=1,\ldots, K$.
\end{definition}

The collection of dyadic families $\mathscr{D}^t$, $t=1,\ldots, K$, obtained by repeating the specific selection rule with all the choices of $(\ell ,m)$ further have the following property:

\begin{lemma}\label{prop;cubeandball}
For every ball $B=B(x,r)\subseteq X$ there exist an integer $t$ and a dyadic cube $Q\in \mathscr{D}^t$ such that
\[ B\subseteq Q\quad\text{and}\quad \diam (Q)\leq C\,r, \]
where $C=C(A_0,\delta)$ is a constant independent of $x$ and $t$.
\end{lemma}
\begin{remark}
Note that the proof will show that for $r$ with $\delta^{k+2}< r \leq \delta^{k+1}$, we may assume that the containing cube $Q\in \mathscr{D}^t$ is of generation $k$. Further, $\rho(x,z_\alpha^k)<\delta^{k+1}$ where $z_\alpha^k$ denotes the center of $Q$. Also note that this is the second part of Theorem~\ref{thm:adjacent;systems}.
\end{remark}
\begin{proof}
Fix $B(x,r)\subseteq X, r>0$, and pick $k\in \Z$ so that $\delta^{k+2}< r \leq \delta^{k+1}$. There exists a reference point $x_\beta^{k+1}$ with duplex label, say $(\ell,m)$, such that
\[\rho (x,x_\beta^{k+1})<\delta^{k+1} . \]
Let $\alpha$ be the unique index for which $(k+1,\beta)\leq (k,\alpha)$. Then $x_{\beta}^{k+1}={} ^t\!z_{\alpha}^{k}$ which is a new dyadic point of generation $k$ in the system $\mathscr{D}^t$, $t=(\ell,m)$. We will prove that $B(x,r)\subseteq B({} ^t\!z_\alpha^{k},c_1\delta^{k})$ where $c_1=(3A_0^2)^{-1}c_0=(12A_0^4)^{-1}$. Indeed, suppose $y\in B(x,r)$. Then
\[\rho (y,{} ^t\!z_\alpha^{k})\leq A_0\rho(y,x)+A_0\rho(x,{} ^t\!z_\alpha^{k})<A_0r+A_0\delta^{k+1}\leq 2A_0\delta^{k+1} \leq c_1\delta^{k}, \]
since $24A_0^5\delta \leq 1$. Thus, $y\in Q:={} ^tQ^{k}_\alpha\in \mathscr{D}^t$, $t=(\ell,m)$. For the diameter of $Q$ we get
\[\diam (Q)\leq \diam \big(B({} ^t\!z_\alpha^{k},C_1\delta^{k})\big) \leq 2A_0C_1\delta^{k} =\frac{2A_0C_1}{\delta^{2}}\delta^{k+2}\leq Cr, \]
with $C:=2A_0C_1\delta^{-2} = 8A_0^3\delta^{-2}$.
\end{proof}

\begin{definition}[Specific selection rule with a distinguished point]
Given $x_0\in X$, recall from \ref{existence;dyadicpoints} that the set of reference points can be chosen in such a way that for every $k\in \Z$ there exists $\alpha$ such that $x^k_\alpha =x_0$. Fix $(\ell,m)\in \{0, \ldots ,L\}\times \{1, \ldots ,M\}$. For every $k$, begin with $x_0=x^k_\alpha$ and decree that $z^k_\alpha :=x_0$. For every $\beta\neq \alpha$, choose the new points $z^k_\beta$ by the specific selection rule, as defined earlier.
\end{definition}

Note that the specific selection rule with a distinguished point is again a special case of the general selection rule. It is not, however, a special case of the specific selection rule. Hence, we need to verify that the dyadic systems, obtained by repeating the specific selection rule with a distinguished point with all the choices of $(\ell,m)$ also satisfy the assertions of Lemma \ref{prop;cubeandball}:

\begin{lemma}\label{dyadic_points_x0} 
Given a fixed point $x_0\in X$, there exist finitely many point sets $\{ z^k_\alpha \}_{k, \alpha}$ such that each of them satisfy the assertions of Lemma~\ref{dyadic_points} and further have the property that for every $k\in \Z$ there exists $\alpha$ such that $z^k_\alpha =x_0$. In addition, the family of dyadic systems defined by these new dyadic points has the property of Lemma~\ref{prop;cubeandball}.
\end{lemma}

\begin{remark}
Note that the proof will show that for $r$ with $\delta^{k+2}< r \leq \delta^{k+1}$, we may assume that the containing cube $Q\in \mathscr{D}^t$ is of generation $k-1$. Further, $\rho(x,z_\alpha^k)<2A_0\delta^{k}$ where $z_\alpha^k$ denotes the center point of $Q$. Also note that Lemma \ref{dyadic_points_x0} completes the proof of Proposition~\ref{prop:adjacent;systems}.
\end{remark} 

\begin{proof}
The assertions of Lemma~\ref{dyadic_points} are clear, since we are still in the regime of the general selection rule. We consider the assertion of Lemma~\ref{prop;cubeandball}. Fix a ball $B=B(x,r)$ in $X$ with $\delta^{k+2}<r\leq \delta^{k+1}$. There exists a reference point $x^{k+1}_\beta$ such that $\rho(x,x^{k+1}_\beta)<\delta^{k+1}$.

Assume first that $x^{k}_\alpha = x_0$ for the unique $(k,\alpha)\geq (k+1,\beta)$. Then also $x^{k-1}_\gamma=x_0$ for the unique $(k-1,\gamma)\geq (k,\alpha)$ implying that $^t\!z^{k-1}_\gamma=x_0$ by the specific selection rule with a distinguished point with every $t=(\ell,m)$. Take $y\in B(x,r)$. Then
\begin{equation*}
\begin{split}
  \rho(y,{} ^t\!z^{k-1}_\gamma) &=\rho(y,x_0)\leq A_0^2\rho(y,x)+A_0^2\rho(x,x^{k+1}_\beta)+A_0\rho(x^{k+1}_\beta,x_0)<A_0^2r+A_0^2\delta^{k+1}+A_0\delta^{k}\\
    &\leq \big(2A_0^2\delta^2+A_0\delta)\delta^{k-1}
      \leq\left(2A_0^2\frac{1}{24^2A_0^{10}}+A_0\frac{1}{24 A_0^5}\right)\delta^{k-1}<\frac{1}{12A_0^4}\delta^{k-1}=c_1\delta^{k-1}
\end{split}
\end{equation*}
since $24A_0^5\delta \leq 1$ and $c_1=(3A_0^2)^{-1}c_0=(12A_0^4)^{-1}$. Thus, $y\in Q:={} ^tQ^{k-1}_\gamma$ for any $t$. For the diameter of $Q$ we have
\[\diam (Q)\leq \diam \big(B({} ^t\!z_\gamma^{k-1},C_1\delta^{k-1})\big) \leq 2A_0C_1\delta^{k-1} =\frac{2A_0C_1}{\delta^{3}}\delta^{k+2}\leq Cr, \]
with $C:=2A_0C_1\delta^{-3} = 8A_0^3\delta^{-3}$.

If $x^{k}_\alpha \neq x_0$ for $(k,\alpha)\geq (k+1,\beta)$, then the new dyadic point ${} ^t\!z^{k}_{\alpha}$ is chosen among the $x^{k+1}_{\sigma}$ with $(k+1,\sigma)\leq(k,\alpha)$ exactly as in the specific selection rule (without a distinguished point). Thus, the reference point $x^{k+1}_\beta={}^t\!z^k_\alpha$ for $t=\lab_{2}(k+1,\beta)$, and the proof is completed by same argument as in Lemma~\ref{prop;cubeandball}.
\end{proof}

\section{Random dyadic systems}
In this section we will prove the following theorem, originally from \cite{HM09}. The present contribution consists of a detailed and streamlined construction of the underlying probability space $\Omega$, the details of which already turned out helpful in an application to singular integrals in \cite{Henri:10}.

\begin{theorem}\label{thm:section5}
Given a set of reference points $\{x^k_\alpha \},k\in \Z, \alpha\in\mathscr{A}_k$, suppose that constant $\delta\in (0,1)$ satisfies $96A_0^6\delta\leq 1$. Then there exists a probability space $(\Omega ,\prob)$ such that every $\omega\in \Omega$ defines a dyadic system $\mathscr{D}(\omega)=\{Q^k_\alpha (\omega)\}_{k, \alpha}$, related to new dyadic points $\{z^k_\alpha (\omega)\}_{k,\alpha}$, with the properties \eqref{eq:open-closed}--\eqref{eq:monotone} of Theorem~\ref{thm:cubes}. Further, the probability space $(\Omega,\prob)$ has the following properties:
{\setlength\arraycolsep{2pt}
\begin{eqnarray}\label{ProbSpace:one}
\quad \Omega &=&\prod_{k\in \Z}\Omega_k; \quad \omega =(\omega_k)_{k\in \Z}\in \Omega \text{ with coordinates $\omega_k\in \Omega_k$ which are independent};\\
\quad z^k_\alpha(\omega) &=& z^k_\alpha (\omega_k)\label{ProbSpace:three};
\end{eqnarray}}
\vspace{-0.3cm}
\begin{equation}\label{ProbSpace:two}
\hspace*{-3cm}\text{if $(k+1,\beta)\leq (k,\alpha)$, then } \prob(\{\omega\in \Omega\colon z^k_\alpha(\omega)=x^{k+1}_\beta \})\geq \tau_0>0.
\end{equation}
\end{theorem}

\begin{remark}
In this chapter we will construct a probability space $(\Omega,\prob)$ by randomizing the choice of new dyadic points from the reference points with respect to all the possible degrees of freedom. The properties \eqref{ProbSpace:one}--\eqref{ProbSpace:two} can, however, be obtained with much less randomness. We will return to this in Section~\ref{sec:adjRandom}. 
\end{remark}

We will first state the following theorem which presents the general properties of all the random dyadic systems with the properties \eqref{ProbSpace:one}--\eqref{ProbSpace:two}. For the slightly different random systems originally constructed in \cite{HM09}, the property~\eqref{property:prob;small} below was already established; its consequences stated as \eqref{property:prob;zero} and \eqref{property:measure;zero} were observed and applied in \cite{HLYY}.

\begin{theorem}\label{probabilistic:statements}
Given a set of reference points $\{x^k_\alpha \},k\in \Z, \alpha\in\mathscr{A}_k$, suppose that constant $\delta\in (0,1)$ satisfies $144A_0^8\delta\leq 1$. Suppose $(\Omega,\prob)$ is any probability space such that every $\omega\in \Omega$ defines a dyadic system $\mathscr{D}(\omega)=\{Q^k_\alpha (\omega)\}_{k, \alpha}$, related to new dyadic points $\{z^k_\alpha (\omega)\}_{k,\alpha}$, with the properties \eqref{eq:open-closed}--\eqref{eq:monotone} of Theorem~\ref{thm:cubes}. Suppose further that the space $(\Omega ,\prob)$ has the properties \eqref{ProbSpace:one}--\eqref{ProbSpace:two} of Theorem~\ref{thm:section5}. Then the following probabilistic statements hold:

For every $x\in X, \tau>0$ and $k\in \Z$, 
\begin{equation}\label{property:prob;small}
\prob \big( \{\omega\in \Omega\colon x\in \bigcup_{\alpha}\partial_{\tau\delta^k}Q^k_\alpha (\omega) \} \big) \leq C_2\tau^\eta 
\text{ for some $C_2,\eta >0$},
\end{equation}
where
\[\partial_\varepsilon Q:=\{x\in \bar{Q}\colon \rho(x,\tilde{Q}^c)\leq \varepsilon \},\; \varepsilon>0;  \]

For every $x\in X$, 
\begin{equation}\label{property:prob;zero}
\prob \big( \{\omega\in \Omega\colon x\in \bigcup_{k,\alpha}\partial Q^k_\alpha (\omega) \} \big) =0;
\end{equation}

Given a positive $\sigma$-finite measure $\mu$ on $X$,
\begin{equation}\label{property:measure;zero}
\mu \big( \bigcup_{k,\alpha}\partial Q^k_\alpha (\omega) \big) =0 \text{ for a.e. $\omega\in \Omega$}.
\end{equation}
\end{theorem}

\subsection{The probability space}\label{prob;space}
Keeping with the fixed set $\{x^k_\alpha\}$, $k\in \Z$, $\alpha\in \mathscr{A}_k$, of reference points, we randomize the construction of new dyadic points from them.
This amounts to formalizing the underlying space of all possible choices of new dyadic points allowed by the general selection rule, and then defining a natural probability measure on this space.
The underlying probability space $\Omega$ will be formed by countable products and unions of finite probability spaces as follows: 
\begin{align*}
\Omega := \prod_{k\in \Z}\Omega_k;  \qquad
\Omega_k: =\bigcup_{\ell_k\in \{0,\ldots ,L \} }
  \Big(\{\ell_k\}\quad &\times\prod_{\alpha\in \mathscr{A}_k\atop \lab_1(k,\alpha)=\ell_k} \{\gamma\colon (k+1,\gamma)\leq (k,\alpha)\} \\
 &\times \prod_{\alpha\in \mathscr{A}_k\atop \lab_1(k,\alpha)\neq\ell_k}
\{\gamma\colon \rho(x^{k+1}_\gamma,x^k_\alpha)<\delta^{k+1}\}\Big).
\end{align*}

For the finite sets $\{\gamma\colon (k+1,\gamma)\leq (k,\alpha)\}$ and $\{\gamma\colon \rho(x^{k+1}_\gamma,x^k_\alpha)<\delta^{k+1}\}$, we use the $\sigma$-algebras consisting of all sub-sets. The $\sigma$-algebra $\mathscr{G}_k$ of the set $\Omega_k$ is the $\sigma$-algebra generated by these sets. We will further consider
\begin{equation*}
 \mathscr{H}_k:= \Big\{ \prod_{j<k}\Omega_j\times G_k \times \prod_{j>k}\Omega_j\colon G_k\in \mathscr{G}_k\Big\} ,\qquad k\in\Z.
\end{equation*}
Then $\sigma$-algebra $\mathscr{H}$ of $\Omega$ is the one generated by the $\sigma$-albegras $\mathscr{H}_k$.

The points $\omega\in\Omega$ admit the natural coordinate representation $\omega=(\omega_k)_{k\in\Z}$, where moreover
\begin{equation*}
  \omega_k=(\ell_k;\omega_{k,\alpha}:\alpha\in\mathscr{A}_k)\in\Omega_k,
\end{equation*}
where $\ell_k\in\{0,\ldots,L\}$ and each $\omega_{k,\alpha}\in\mathscr{A}_{k+1}$ satisfies $(k+1,\omega_{k,\alpha})\leq(k,\alpha)$, as well as $\rho(x^{k+1}_{\omega_{k,\alpha}},x^k_\alpha)<\delta^{k+1}$ if $\lab_1(k,\alpha)\neq\ell_k$.

We define a probability $\prob$ on $\Omega$ by requiring the coordinates $\omega_k$ to be independent and distributed as follows: First,
\begin{equation*}
  \prob(\ell_k=\ell)=\frac{1}{L+1}\qquad \forall\, \ell=0,\ldots,L.
\end{equation*}
Second, given the master label $\ell_k$, the subcoordinates $\omega_{k,\alpha}$, $\alpha\in\mathscr{A}_k$, are again independent, with distribution
\begin{align*}
  \prob(\omega_{k,\alpha} &=\beta |\ell_k=\ell) \\
  &=\begin{cases}
    [\#\{ \gamma\colon (k+1,\gamma)\leq (k,\alpha)\}]^{-1}\quad \forall (k+1,\beta)\leq (k,\alpha), &\text{if } \lab_1(k,\alpha)=\ell_k, \\
    [\#\{ \gamma\colon \rho(x^{k+1}_\gamma ,x^k_\alpha)<\delta^{k+1}\}]^{-1}\quad \!\!\forall(k+1,\beta):\rho(x^{k+1}_\beta ,x^k_\alpha)<\delta^{k+1}, &\text{if } \lab_1(k,\alpha)\neq \ell_k.
   \end{cases}
\end{align*}

Note that there is an obvious one-to-one correspondence between the coordinates $\omega_k\in\Omega_k$ and the admissible choices of the relation $\searrow$ between index pairs on levels $k$ and $k+1$, subject to the general selection rule. This relation in turn uniquely determines the new dyadic points $z^k_{\alpha}=z^k_{\alpha}(\omega_k)$ for the given level $k\in\Z$, and thus the choice of $\omega=(\omega_k)_{k\in\Z}$ uniquely determines the new dyadic points $z^k_{\alpha}=z^k_{\alpha}(\omega)$ on all levels $k\in\Z$. By a random choice of the new dyadic points, we understand the new dyadic points $z^k_{\alpha}(\omega)$, where $\omega\in\Omega$ is distributed according to the probability $\prob$.

Once the points $z^k_{\alpha}(\omega)$ are chosen, they uniquely determine the relation $\leq_{\omega}$ between the index pairs $(k,\alpha)$ (not to be confused with the original relation $\leq$, which is in general not the same); recall that it is possible to make the choice of $\leq_{\omega}$ in such a way that it only depends on the dyadic points without any arbitrariness. Then the points $z^k_{\alpha}(\omega)$ and the relation $\leq_{\omega}$ together determine the new dyadic cubes $Q^k_{\alpha}(\omega)$ as a function of $\omega\in\Omega$, and their random choice corresponds to the random choice of $\omega$ according to the law $\prob$.

It is evident, by Lemma~\ref{dyadic_points}, that for every $\omega\in\Omega$, the dyadic system $\mathscr{D}(\omega)$ satisfies the properties \eqref{eq:open-closed}--\eqref{eq:monotone} of Theorem~\ref{thm:cubes}.

Note that by construction, for every $(k+1,\beta)\leq (k,\alpha)$,
\[\prob (\{\omega\in \Omega\colon z^k_\alpha(\omega)=x^{k+1}_\beta\})\geq [(L+1)\#\{\gamma\colon (k+1,\gamma)\leq (k,\alpha)\}]^{-1}\geq [(L+1)M]^{-1}=:\tau_0>0.\]
This completes the proof of Theorem~\ref{thm:section5}.

\subsection{A technical lemma}

Before turning to a more thorough investigation of the random dyadic cubes as just defined, we provide a technical lemma, which has nothing to do with the randomness, but is a general property of all dyadic systems. However, we will only make use of this lemma in the randomized context, which is the reason of including it in this section. Roughly speaking, the lemma states that in order to reach the boundary of a cube from its centre, along a direct line of ancestry of dyadic points, one needs to make jumps of non-trivial size at every step. This result goes back to Christ \cite{Christ90}, and appeared as part of the proof of his Lemma~17. Christ's lemma concerned the smallness of the boundary region of the dyadic cubes with respect to an underlying doubling measure; the technical intermediate result is valid even without the presence of a measure, and we will apply it to get analogous smallness results for the boundary with respect to the probability $\prob$ defined above.

\begin{lemma}\label{lem:technical}
Suppose that constants $0<c_0\leq C_0 <\infty$ and $\delta \in (0,1)$ satisfy $18A_0^5C_0\delta \leq c_0$. Let $\{z^k_{\alpha}\}_{k,\alpha}$ be a set of points as in Theorem~\ref{thm:cubes}. Given $N\in\Z_+$ and $\tau >0$, suppose that $12A_0^4\tau\leq c_0\delta^N$. Let $x\in\bar{Q}^k_{\alpha}$ with $\rho(x,(\tilde{Q}^k_{\alpha})^c)<\tau\delta^k$. For all chains
\begin{equation*}
  (k+N,\sigma)=(k+N,\sigma_{k+N})
  \leq\ldots\leq (k+1,\sigma_{k+1})\leq (k,\sigma_k)
\end{equation*}
such that $x\in\bar{Q}^{k+N}_{\sigma}$, there holds $\rho(z^j_{\sigma_j},z^i_{\sigma_i})\geq\eps_1\delta^j$, $\eps_1:=(12A_0^4)^{-1}c_0$, for all $k\leq j<i\leq k+N$.
\end{lemma}

\begin{proof}
Let $(k,\alpha)$ be fixed and consider $x\in \bar{Q}_\alpha^k$ with $\rho(x,(\tilde{Q}^k_{\alpha})^c)<\tau\delta^k$ for some $\tau >0$. 
Let $(j,\sigma_j)$ be the intermediate pairs as in the assertion, and abbreviate $z^j:=z^j_{\sigma_j}$ for $k\leq j\leq k+N$. Suppose for contradiction that $\rho(z^j,z^i)<\eps_1\delta^j$ for some $k\leq j<i\leq k+N$. There are two possibilities: $\sigma_k=\alpha$ or not.

First assume $\sigma_k=\alpha$ (i.e. the chain travels in $Q_\alpha^k$). Then, as $x\in\bar{Q}^{k+N}_{\sigma}\subset\bar{Q}^{i}_{\sigma_i}$, also $x\in B(z^i,C_1\delta^i)$ for $(i,\sigma_i)\geq (k+N,\sigma)$. We also have $B(z^j,c_1\delta^j)\subseteq\tilde{Q}^j_{\sigma_j}\subseteq\tilde{Q}_\alpha^k$, and so it follows that
\begin{equation*}
\begin{split}
  c_1\delta^j\leq\rho(z^j,(\tilde{Q}_\alpha^k)^c)
  &\leq A_0\rho(x,(\tilde{Q}_\alpha^k)^c)+A_0^2\rho(x,z^i)+A_0^2\rho(z^i,z^j) \\
  &< A_0\tau\delta^k+A_0^2 C_1\delta^i+A_0^2\eps_1\delta^j
  \leq \frac{1}{4}c_1\delta^{N+k}+\frac{1}{3}c_1\delta^{i-1}+\frac{1}{4}c_1\delta^{j}
  \leq c_1\delta^j,
\end{split}
\end{equation*}
since $c_1:=(3A_0^2)^{-1}c_0$, $C_1:=2A_0C_0$, $4A_0^2\tau\leq c_1\delta^N$, $3A_0^2 C_1\delta\leq c_1$ and $4A_0^2\eps_1\leq c_1$, and this is a contradiction.

If $\sigma_k\neq\alpha$ (and the chain travels outside $Q_\alpha^k$), we have $x\in\bar{Q}^{k+N}_{\sigma}\subseteq\bar{Q}^k_{\sigma_k}$ and $\rho(x,(\tilde{Q}^k_{\sigma_k})^c)=0<\tau\delta^k$. Thus we are in the identical situation as before but with $\sigma_k$ in place of $\alpha$. Hence the same conclusion applies.
\end{proof}

\subsection{The proof of Theorem~\ref{probabilistic:statements}} From now on, assume $(\Omega,\prob)$ is a probability space with the properties \eqref{ProbSpace:one}--\eqref{ProbSpace:two} of Theorem~\ref{thm:section5}, and that $144A_0^8\delta\leq 1$.

\begin{definition}[Boundary zone of a dyadic cube]
For $\varepsilon >0$, we denote
\[\partial_{\varepsilon}Q:=\{x\in \bar{Q}\colon \rho(x,\tilde{Q}^c)\leq \varepsilon \}. \]
\end{definition}

We mention that if the space $(X,\rho)$ supports a doubling measure $\mu$, Lemma~\ref{lem:technical} has the following consequence \cite[Lemma~17]{Christ90}: For every $\varepsilon >0$ there exists $\tau \in (0,1]$ such that for every dyadic cube $Q_\alpha^k$,
\begin{equation*}
  \mu(\partial_{\tau\delta^k}Q^k_{\alpha})<\eps\mu(Q^k_{\alpha}).
\end{equation*}
Here, we are concerned with the following probabilistic analogue:

\begin{lemma}[\eqref{property:prob;small} of Theorem~\ref{probabilistic:statements}]\label{lem:boundary}
For a given $x\in X$ and $\tau>0$ and a fixed $k\in \Z$, there holds 
\[\prob\Big(\Big\{ \omega\in \Omega\colon x\in \bigcup_{\alpha}\partial_{\tau\delta^k}Q^k_\alpha(\omega) \Big\} \Big)\leq C_2\tau^{\eta}
\] 
for some constants $C_2,\eta>0$. 
\end{lemma}

\subsection{Reduction}\label{reduction}
We consider the event
\[E= \Big\{ \omega\in \Omega\colon x\in \bigcup_{\alpha}\partial_{\tau\delta^k}Q^k_\alpha(\omega) \Big\} .\]
First note that for every $x\in X$ and $k\in \Z$, there exists a finite set $A=A_k(x)$ of indices such that if $x\in \bar{Q}^k_\alpha(\omega)$ for any $\omega\in \Omega$, it follows that $\alpha\in A_k(x)$. Moreover, $\# A_k(x)\leq C<\infty$ where $C$ is independent of $x$ and $k$. Indeed, if $x\in \bar{Q}^k_\alpha(\omega)$ we have $\rho(x,z^k_\alpha(\omega))<C_1\delta^k$, and thus
\[\rho(x,x^k_\alpha)\leq A_0\rho(x,z^k_\alpha(\omega))+A_0\rho(z^k_\alpha(\omega),x^k_\alpha) <A_0(C_1+1)\delta^k,\]
since $\rho(z^k_\alpha(\omega),x^k_\alpha)<\delta^k$ by the choice of $z^k_\alpha(\omega)$. By the geometric doubling property, the ball $B(x,A_0(C_1+1)\delta^k)$ can contain at most boundedly many centres $x^k_\alpha$ of the disjoint balls
\begin{equation*}
  B(x^k_\alpha,(2A_0)^{-1}\delta^k).  
\end{equation*}
 
In particular, if $x\in \bigcup_{\alpha}\partial_{\tau\delta^k}Q^k_\alpha(\omega) $, then $x\in \bigcup_{\alpha\in A_k(x)}\partial_{\tau\delta^k}Q^k_\alpha(\omega) $ where $\# A_k(x)\leq C<\infty$ and $C$ is independent of $x$ and $k$. Since the closed dyadic cubes of any generation $k+N$ cover $X$, it follows that
 \begin{equation*}
 \begin{split}
 E &=  \Big\{ \omega\in \Omega\colon x\in \Big(\bigcup_{\alpha}\partial_{\tau\delta^k}Q^k_\alpha(\omega)\Big) \cap \Big(\bigcup_{\sigma}\bar{Q}^{k+N}_\sigma(\omega)\Big)\Big\}\\
 & =\Big\{ \omega\in \Omega\colon x\in 
 \Big(\bigcup_{\alpha\in A_k(x)}\partial_{\tau\delta^k}Q^k_\alpha(\omega)\Big) \cap \Big(\bigcup_{\sigma\in A_{k+N}(x)}\bar{Q}^{k+N}_\sigma(\omega)\Big) \Big\}\\
 &=\Big\{ \omega\in \Omega\colon x\in 
 \bigcup_{\alpha\in A_k(x)\atop \sigma\in A_{k+N}(x)} \Big(\partial_{\tau\delta^k}Q^k_\alpha(\omega) \cap \bar{Q}^{k+N}_\sigma(\omega)\Big) \Big\}.
 \end{split}
 \end{equation*}
Here the union is bounded, and we have
\[\prob (E)\leq \sum_{\alpha\in A_k(x)\atop \sigma\in A_{k+N}(x)}\prob \left(\left\{ \omega\in \Omega\colon x\in \partial_{\tau\delta^k}Q^k_\alpha(\omega)\cap \bar{Q}^{k+N}_\sigma(\omega) \right\} \right). \]
Thus, in order to prove Lemma~\ref{lem:boundary}, it suffices to prove that 
\[\prob\left(\left\{ \omega\in \Omega\colon x\in \partial_{\tau\delta^k}Q^k_\alpha(\omega)\cap \bar{Q}^{k+N}_\sigma(\omega) \right\} \right)\leq C_2\tau^{\eta}
\] 
for some constants $C_2,\eta>0$ with fixed $x\in X$, $\tau>0$, $N\in \N$, $k\in \Z$, $\alpha$ and $\sigma$.

We first state the following basic probability lemma, which we have included for the convenience of readers less experienced with conditional expectations; this is essentially the only place where probabilistic reasoning beyond standard measure theory will be needed. 

\begin{lemma}\label{lem:conditional;exp}
Let $\{\mathscr{F}_j\}$, $j=1, \ldots , k$, be a finite collection of $\sigma$-algebras and suppose that $\mathscr{F}_{j+1}\subseteq\mathscr{F}_j$ and $A_j\in\mathscr{F}_j$ for all $j$. Then
\[\Exp\prod_{j=1}^{k-1}\chi_{A_j}=\Exp
\Exp_{\mathscr{F}_{k}}\chi_{A_{k-1}}\Exp_{\mathscr{F}_{k-1}}\chi_{A_{k-2}}\ldots
\Exp_{\mathscr{F}_{2}}\chi_{A_{1}} \]
where $\Exp_{\mathscr{F}_j}[\,\cdot\,]:=\Exp [\;  \cdot \; \vert \mathscr{F}_j]$ denotes the conditional expectation given $\mathscr{F}_j$.
\end{lemma}

\begin{proof}
By the properties of conditional expectation, see for example \cite{williams}, \S 9.7,
\[
\Exp \Big(\prod_{j=1}^{k-1}\chi_{A_j}\Big) 
=\Exp\, \Big(\Exp_{\mathscr{F}_{k}} \Big[\prod_{j=1}^{k-1}\chi_{A_j}\Big]\Big) .
\]
First we use the so-called Tower Property: Since $\mathscr{F}_{k}\subseteq \mathscr{F}_{k-1}$, there holds
\begin{equation}\label{first;step}
\Exp_{\mathscr{F}_{k}} \Big[\prod_{j=1}^{k-1}\chi_{A_j}\Big]
=  \Exp_{\mathscr{F}_{k}} \Big[\Exp_{\mathscr{F}_{k-1}}\Big[\prod_{j=1}^{k-1}\chi_{A_j}\Big]\Big] .
\end{equation}
Secondly, the fact that $\chi_{A_{k-1}}$ is $\mathscr{F}_{k-1}$-measurable implies that
\begin{equation}\label{second;step}
\Exp_{\mathscr{F}_{k-1}}\Big[\prod_{j=1}^{k-1}\chi_{A_j}\Big]
= \chi_{A_{k-1}}\Exp_{\mathscr{F}_{k-1}} \Big[\prod_{j=1}^{k-2}\chi_{A_j}\Big] .
\end{equation}
We now repeat steps \eqref{first;step} and \eqref{second;step} $k-1$ times to conclude that 
\[\Exp \Big(\prod_{j=1}^{k-1}\chi_{A_j}\Big) = 
\Exp\big(\Exp_{\mathscr{F}_k}[\chi_{A_{k-1}}\Exp_{\mathscr{F}_{k-1}}[\chi_{A_{k-2}}\ldots 
\chi_{A_{2}}\Exp_{\mathscr{F}_{2}}[\chi_{A_1}]\ldots]\big).\qedhere \]
\end{proof}

\begin{proof}[Proof of Lemma~\ref{lem:boundary}]
Fix $x\in X$ and $k\in \Z$.  Given $\tau\in(0,(4A_0^2)^{-1}c_1)$, pick the unique $N\in\N:=\{0,1,\ldots \}$ so that $c_1\delta^{N+1}< 4A_0^2\tau\leq c_1\delta^N$. (Since any probability is at most $1$, the claim is of course true for any $\eta$ when $\tau\geq (4A_0^2)^{-1}c_1$, taking large enough $C_2$.) Note that, in particular, $12A_0^4\tau \leq c_0\delta^N$ since $c_1:=(3A_0^2)^{-1}c_0$. Also note that, by the assumption $144A_0^8\delta\leq 1$, we have $18A_0^5C_0\delta\leq c_0$ since $c_0=(4A_0^2)^{-1}$ and $C_0=2A_0$. Thus, the parameter assumptions of Lemma~\ref{lem:technical} hold. By the reduction in \ref{reduction}, it suffices to consider the event
\[E_{\alpha,\sigma}:=\Big\{ \omega\in \Omega\colon x\in \partial_{\tau\delta^k}Q^k_\alpha(\omega)\cap \bar{Q}^{k+N}_\sigma(\omega) \Big\} \]
for fixed $\alpha$ and $\sigma$.

Let $\sigma=:\sigma_{k+N}$. For every $j=k, k+1,\ldots ,k+N-1$, let us denote
\[A_j:=\{\omega\in \Omega \colon \rho(z^j_{\sigma_j}(\omega),z^{j+1}_{\sigma_{j+1}}(\omega))\geq\varepsilon_1\delta^j \text{ for } (j,\sigma_j)\geq_{\omega} (j+1,\sigma_{j+1})\geq_{\omega}(k+N,\sigma)\} \]
where $\varepsilon_1:=(2A_0)^{-2}c_1$ is the constant from Lemma~\ref{lem:technical}. Note that the sets $A_j$ only depend on the choice of points of levels from $k+N$ to $j$ and, by \eqref{ProbSpace:three}, the choice of these points only depend on $\Omega_j$ for $j=k, k+1,\ldots ,k+N-1$.
By Lemma~\ref{lem:technical}, it particularly holds that
\[E_{\alpha,\sigma}\subseteq \bigcap_{j=k}^{k+N-1}A_j .\]
Let us denote by 
\[\mathscr{F}_j := \sigma \Big(\mathscr{H}_i\colon i\geq j \Big),\]
the $\sigma$-algebra generated by the class of subsets of $\Omega$ with the points of level $i\geq j$ fixed. Note that 
\begin{equation*}
A_j\in \mathscr{F}_j\text{ for every $j=1, \ldots ,N$ and } \mathscr{F}_{j+1}\subseteq \mathscr{F}_{j}\text{ for every }j=k,\ldots, k+N-1.
\end{equation*}
By Lemma~\ref{lem:conditional;exp}, we have
\begin{equation}\label{eq:the;chain}
\begin{split}
\prob(E_{\alpha ,\sigma}) & =\Exp(\chi_{E_{\alpha ,\sigma}})\leq \Exp\Big(\prod_{j=k}^{k+N-1}\chi_{A_j}\Big)\\
& =\Exp
\Exp_{\mathscr{F}_{k+N}}\chi_{A_{k+N-1}}\Exp_{\mathscr{F}_{k+N-1}}\chi_{A_{k+N-2}}\ldots
\Exp_{\mathscr{F}_{k+1}}\chi_{A_{k}} .
\end{split}
\end{equation}
We first calculate 
\[\Exp_{\mathscr{F}_{k+1}} [\chi_{A_k}] = \prob ( A_k\vert \mathscr{F}_{k+1} ) .\]
Note that for a given index pair $(k+1,\sigma_{k+1})$, there always exists a reference point $x^{k+1}_\gamma$ such that $\rho(x^{k+1}_\gamma,z^{k+1}_{\sigma_{k+1}}(\omega))<\delta^{k+1}<\varepsilon_1\delta^k$. On the other hand, by \eqref{ProbSpace:two}, there is a positive probability $\tau_0$ that $x^{k+1}_\gamma=z^k_{\gamma}(\omega)$. Thus, for a given index pair $(k+1,\sigma_{k+1})$, there is a positive probability that the pair $(k,\sigma_k)$ for which $(k,\sigma_k)\geq_\omega (k+1,\sigma_{k+1})$ satisfies $\rho(z^{k}_{\sigma_{k}},z^{k+1}_{\sigma_{k+1}})<\varepsilon_1\delta^k$. Since the negation for the event $A_k$ with given $\mathscr{F}_{k+1}$ is that for some index pair $(k+1,\sigma_{k+1})$, the parent is within the distance $\varepsilon_1\delta^k$, we conclude with
\begin{equation}\label{third;step}
\Exp_{\mathscr{F}_{k+1}} \chi_{A_k} = \prob ( A_k\vert \mathscr{F}_{k+1} )\leq 1-\tau_0 ,\quad \tau_0>0.
\end{equation}
Further, we have by the above, monotonicity and linearity
\begin{equation}\label{fourth;step}
\Exp_{\mathscr{F}_{k+2}}\Big[\chi_{A_{k+1}}\Exp_{\mathscr{F}_{k+1}} \chi_{A_k}\Big]
\leq  \Exp_{\mathscr{F}_{k+2}}\Big[\chi_{A_{k+1}}(1-\tau_0) \Big] =(1-\tau_0)\Exp_{\mathscr{F}_{k+2}}[\chi_{A_{k+1}}]
\end{equation}
since $1-\tau_0$ is a constant. We now proceed backwards and travel from the end of the chain in \eqref{eq:the;chain} repeating the steps in \eqref{third;step}  and \eqref{fourth;step} $N-1$ times. Each time the term $\Exp_{\mathscr{F}_{k+i}} \chi_{A_i}$, $i=1, \ldots , N$ is estimated from above by constant $1-\tau_0\in (0,1)$, which can then be relocated by equation \eqref{fourth;step}. What is obtained is the following:
\[ \prob (E_{\alpha,\sigma})\leq (1-\tau_0)^{N} <C_2\tau^\eta \]
with $C_2:=4A_0^2(c_1\delta)^{-1}$ and $\eta:=\log (1-\tau_0)/\log \delta>0$.
\end{proof}

\begin{corollary}[\eqref{property:prob;zero} of Theorem~\ref{probabilistic:statements}]\label{cor:pb;bdr:zero}
For $x\in X$,
\[\prob\Big(\Big\{ \omega\in \Omega\colon x\in \cup_{\alpha ,k}\partial Q^k_\alpha(\omega) \Big\} \Big)=0.
\] 
\end{corollary}
\begin{proof}
Recall from \eqref{eq:open-closed} that the cubes $\bar{Q}^k_\alpha$ and $\tilde{Q}^k_\alpha$ are the interior and closure of $Q^k_\alpha$, respectively. Thus,
\[\partial Q^k_\alpha=\bar{Q}^k_\alpha \setminus \tilde {Q}^k_\alpha= \bar{Q}^k_\alpha\cap (\tilde {Q}^k_\alpha)^c \subseteq \{x\in\bar{Q}^k_\alpha\colon \rho(x,(\tilde {Q}^k_\alpha)^c)=0  \}.\]
It follows that, for any $\tau>0$ there holds
\[\partial Q^k_\alpha\subseteq \{x\in\bar{Q}^k_\alpha\colon \rho(x,(\tilde {Q}^k_\alpha)^c)\leq \tau\delta^k  \}=\partial_{\tau\delta^k}Q^k_\alpha.\]
Thus,
\[\Big\{ \omega\in \Omega\colon x\in \cup_{\alpha}\partial Q^k_\alpha(\omega) \Big\}\subseteq
\Big\{ \omega\in \Omega\colon x\in \cup_{\alpha}\partial_{\tau\delta^k} Q^k_\alpha(\omega) \Big\} ,\]
and consequently, by Lemma~\ref{lem:boundary},
\[\prob\Big(\Big\{ \omega\in \Omega\colon x\in \cup_{\alpha}\partial Q^k_\alpha(\omega) \Big\}\Big)\leq
\prob\Big(\Big\{ \omega\in \Omega\colon x\in \cup_{\alpha}\partial_{\tau\delta^k} Q^k_\alpha(\omega) \Big\}\Big) \leq C_2\tau^\eta .\]
Thus, by passing $\tau$ to zero we obtain
\[\prob\Big(\Big\{ \omega\in \Omega\colon x\in \cup_{\alpha}\partial Q^k_\alpha(\omega) \Big\}\Big)=0.\]
Finally,
\[\prob\Big(\Big\{ \omega\in \Omega\colon x\in \cup_{\alpha ,k}\partial Q^k_\alpha(\omega) \Big\}\Big)\leq \sum_{k}\prob\Big(\Big\{ \omega\in \Omega\colon x\in \cup_{\alpha}\partial Q^k_\alpha(\omega) \Big\}\Big)=0.\]
\end{proof}

\begin{lemma}[\eqref{property:measure;zero} of Theorem~\ref{probabilistic:statements}]
Assume that $\mu$ is a positive $sigma$-finite measure on $X$. Then
\begin{equation}\label{boundary:zero}
\mu\left( \cup_{\alpha ,k}\partial Q^k_\alpha (\omega) \right) = 0\quad\text{for a.e. $\omega\in\Omega$}.
\end{equation}
In particular, given $\mu$ we may choose $\omega\in \Omega$ such that \eqref{boundary:zero} holds. 
\end{lemma}
\begin{proof}
For a fixed $\omega\in\Omega$, denote
\[B_\omega:=\bigcup_{\alpha ,k}\partial Q^k_\alpha (\omega),\]
and for a fixed $x\in X$, denote
\[B^x:=\{\omega\in\Omega\colon x\in B_\omega \}. \]
By Fubini's Theorem,
\begin{equation*}
\begin{split}
\int_{\Omega}\mu(B_\omega)\,d\prob (\omega) &= \int_{\Omega}\int_{X}\chi_{B_\omega}(x)\,d\mu(x)d\prob (\omega) \\
&=\int_{X}\int_{\Omega}\chi_{B^x}(\omega)\,d\prob (\omega) d\mu(x)
=\int_{X}\prob(B^x)\, d\mu(x)=0
\end{split}
\end{equation*}
since $\prob (B^x)=0$ by Corollary~\ref{cor:pb;bdr:zero}. The assertion follows.
\end{proof}


\section{Random adjacent dyadic systems}\label{sec:adjRandom}
In this section we will prove the following theorem.

\begin{theorem}\label{thm:section6}
Given a set of reference points $\{x^k_\alpha\}_{k,\alpha}$, suppose that constant $\delta\in (0,1)$ satisfies $144A_0^8\delta\leq 1$. Then there exists a probability space $(\Omega ,\prob)$ such that every $\omega\in \Omega$ defines a family of dyadic systems $(\mathscr{D}^t(\omega))_{t=1}^{K}$ where each $\mathscr{D}^t(\omega)=\{^tQ^k_\alpha (\omega)\}_{k, \alpha}$, related to new dyadic points $\{^tz^k_\alpha (\omega)\}_{k,\alpha}$, satisfies the properties (\ref{eq:open-closed})--(\ref{eq:monotone}) of Theorem~\ref{thm:cubes}. Further,
{\setlength\arraycolsep{2pt}
\begin{eqnarray}
\text{for every $\omega\in \Omega$,}& (\mathscr{D}^t(\omega))_{t=1}^{K}& \text{satisfies the property of Lemma~\ref{prop;cubeandball};} \\
\text{$\quad\quad$for every $t\in \{1,\ldots ,K \}$,}& (\mathscr{D}^t(\omega))_{\omega\in \Omega}& \text{satisfies the properties (\ref{ProbSpace:one})--(\ref{ProbSpace:two}) of Theorem~\ref{thm:section5}}.\label{eq:forallt}
\end{eqnarray}}
\end{theorem}

Notice that the statement of this result makes no reference to randomness but the proof does, and it is not clear how to prove something like this without the help of randomization. In more classical set-ups, similar conclusions could be reached with the help of strongly Euclidean devices like rotations; cf.~\cite{MMNO}, Theorem~2 and its proof.

One immediate application of such a construction is the following. 
\begin{corollary}
Given a set of reference points $\{x^k_\alpha\}_{k,\alpha}$, suppose that constant $\delta\in (0,1)$ satisfies $144A_0^8\delta\leq 1$. Let $\mu$ be a positive $\sigma$-finite measure on $X$. Then the finite collection of adjacent dyadic systems $\mathscr{D}^t$, $t=1,\ldots,K$, as in Theorem~\ref{thm:adjacent;systems}, may be chosen to satisfy the additional property that
\begin{equation*}
  \mu(\partial Q)=0\qquad\forall Q\in\bigcup_{t=1}^K\mathscr{D}^t.
\end{equation*}
\end{corollary}

\begin{proof}
Let $\mathscr{D}^t(\omega)$ be the random adjacent systems guaranteed by Theorem~\ref{thm:section6}. By \eqref{eq:forallt} and \eqref{property:measure;zero} of Theorem~\ref{probabilistic:statements}, we have that
\begin{equation*}
  \forall t\in\{1,\ldots,K\},\text{ for a.e. }\omega\in\Omega,\quad\mu\Big(\bigcup_{Q\in\mathscr{D}^t(\omega)}\partial Q\Big)=0.
\end{equation*}
Since there are only finitely many choices of $t$, we can reverse the order of ``$\forall t$'' and ``for a.e. $\omega\in\Omega$'' above, and then it suffices to choose any $\omega\in\Omega$ outside the event of probability zero implicit in the ``a.e.'', and take $\mathscr{D}^t:=\mathscr{D}^t(\omega)$ for this chosen $\omega\in\Omega$.
\end{proof}

\subsection{The probability space}
We define a probability space $\Omega$ by setting
\[\Omega:=\prod_{k\in \Z} \Omega_k,\qquad \Omega_k:=\{1,2,\ldots ,K\}. \]
The points $\omega\in\Omega$ admit the natural coordinate representation $\omega=(\omega_k)_{k\in \Z}$ where $\omega_k\in \{1,2,\ldots ,K\}$. 

We define a probability $\prob$ on $\Omega$ by requiring the coordinates $\omega_k$ to be independent and distributed with equal probabilities
\[\prob(\omega_k=T)=\frac{1}{K}\quad\forall\, T=1,2,\ldots ,K. \]
Given $\omega_k=T_k$, define a permutation of $(1,2,\ldots ,K )$ by
\[\pi_k(t):=t+T_k\quad (\!\!\!\!\!\!\mod K),\quad t=1,2,\ldots ,K.\]
Given $T_k,t\in\{ 1,2,\ldots ,K\}$, they together define an ordered pair $\pi_k(t)=(\ell_k(t) ,m_k(t))$ via the bijection $\varphi$ introduced in \ref{spesific;rule}. We define the new dyadic points $\{^t\!z^k_\alpha\}_{\alpha}$ of generation $k$ as follows. For every $(k,\alpha)$, check whether there exists $(k+1,\beta)\leq (k,\alpha)$ such that $\lab_2(k+1,\beta)=\pi_k(t)$. If so, decree $^t\!z^k_\alpha :=x^{k+1}_\beta$. Otherwise, pick any $(k+1,\beta)\leq (k,\alpha)$ with $\rho(x^{k+1}_\beta,x^k_\alpha)<\delta^{k+1}$ and decree $^t\!z^k_\alpha :=x^{k+1}_\beta$. (We could, for example, choose the nearest child of $x^k_\alpha$ and eliminate the arbitrariness of this choice.)

The choice of $\omega=(\omega_k)$ then uniquely determines the permutation $\pi_k$ on every level $k$ which in turn determines a family of new dyadic points $^t\!z^k_\alpha(\omega)= \,^t\!z^k_\alpha(\omega_k)=z^k_\alpha (\pi_k(t))$ for each $t=1,\ldots ,K$. 

Once the points $^t\!z^k_{\alpha}(\omega_k)$, $t=1,\ldots K$, are chosen, they uniquely determine the relation $\leq_{\omega,t}$ between the index pairs $(k,\alpha)$. Then, for every $t=1,\ldots,K$, the points $^t\!z^k_{\alpha}(\omega_k)$ and the relation $\leq_{\omega,t}$ together determine the new dyadic cubes $^t\!Q^k_\alpha(\omega_k)$, and their random choice corresponds to the random choice of $\omega$ according to the law $\prob$. Note that the choice of the new dyadic points $^t\!z^k_\alpha(\omega)=z^k_\alpha (\pi_k(t))$ coincides with the specific selection rule defined in \ref{spesific;rule}.

It is evident, by Lemma~\ref{dyadic_points} in view of \ref{spesific;rule}, that for every $\omega\in\Omega$ and every $t=1,\ldots K$, the dyadic system $\mathscr{D}^t(\omega)$ satisfies the properties \eqref{eq:open-closed}--\eqref{eq:monotone} of Theorem~\ref{thm:cubes}. We may complete the proof of Theorem~\ref{thm:section6} by the following lemma: 

\begin{lemma}
For every $\omega\in \Omega$, the family $(\mathscr{D}^t(\omega))_{t=1}^{K}$ satisfies the property of Lemma~\ref{prop;cubeandball}. For every $t=1,\ldots ,K$, $(\mathscr{D}^t(\omega))_{\omega\in \Omega}$ satisfies the properties (\ref{ProbSpace:one})--(\ref{ProbSpace:two}) of Theorem~\ref{thm:section5}.
\end{lemma}
\begin{proof}
Suppose $\omega\in\Omega$, $\omega=(\omega_k)_{k\in \Z}$, and let $\pi_k$ be the permutation defined by $T_k:=\omega_k$. Going back to the proof of Lemma~\ref{prop;cubeandball}, it suffices to prove that for every $(k+1,\beta)$, there holds $x^{k+1}_\beta={}^t\!z^k_\alpha$ for $(k,\alpha)\geq (k+1,\beta)$ and some $t=1,\ldots ,K$. But this is clear since $\lab_2(k+1,\beta)=:t=\pi_k(t-T_k)$ where $t-T_k$ is defined modulo $K$. By construction, $^t\!z^k_\alpha=x^{k+1}_\beta$, and the first assertion follows.

With fixed $t\in\{ 1,\ldots ,K\}$ and $(k+1,\beta)\leq (k,\alpha)$, there is a positive probability $\tau_0= 1/K$ that $\pi_k(t)=s$ for $s=\lab_2(k+1,\beta)$. Thus,
\[\prob(\{\omega\in\Omega\colon ^t\!z^k_\alpha(\omega)=x^{k+1}_\beta\}) =\prob(\{\omega\in\Omega\colon \pi_k(t)=\lab_2(k+1,\beta) \})=\tau_0>0, \]
and hence, (\ref{ProbSpace:two}) of Theorem~\ref{thm:section5} holds. The other properties follow directly from the construction of $(\Omega,\prob)$. The second assertion follows.
\end{proof}
\begin{remark}
By the construction of the probability space $(\Omega ,\prob)$, while the random choice of new dyadic points is independent on different levels, on a given level $k$ it only depends on the choice of $\omega_k$. Thus, the choice of points $\{^t\!z^k_\alpha\}_\alpha$ is not independent. By slightly changing the construction of $\Omega_k$, we may obtain independence also among the choice of non-neighbouring points on the same level: We define a probability space $\Omega$ by setting
\[\Omega:=\prod_{k\in \Z} \Omega_k,\qquad \Omega_k:=\{1,2,\ldots ,K\}\times \prod_{\alpha\in\mathscr{A}_k}\{1,2,\ldots ,M_{k,\alpha}\}, \]
where $M_{k,\alpha}:=\# \{\gamma\colon (k+1,\gamma)\leq (k,\alpha)\}$, the number of children of reference point $(k,\alpha)$.

The points $\omega\in\Omega$ again admit the natural coordinate representation $\omega=(\omega_k)_{k\in \Z}$. Moreover, $\omega_k=(T_k,m_{k,\alpha}\colon \alpha\in \mathscr{A}_k)\in \Omega_k$ where $T_k\in \{1,\ldots ,K \}$ and $m_{k,\alpha}\in \{1,\ldots ,M_{k,\alpha}\}$.

We define a probability $\prob$ on $\Omega$ by requiring the coordinates $\omega_k$ to be independent and distributed as follows. First,
\[\prob(T_k=T)=\frac{1}{K}\quad \forall\, T=1,2,\ldots ,K. \]
Second, the subcoordinates $m_{k,\alpha}$ are again independent with distribution
\[\prob(m_{k,\alpha}=m)=\frac{1}{M_{k,\alpha}}\quad \forall\, m=1,2,\ldots ,M_{k,\alpha}.\]
Given $\omega_k=(T_k,m_{k,\alpha}\colon \alpha\in\mathscr{A}_k)$, we define the new dyadic points as follows. First, $T_k$ defines a cyclic permutation $\pi_k$ of $(1,2,\ldots ,K )$ as before. Then, for every $t=1,\ldots, K$ and $(k,\alpha)$, check whether there exists $(k+1,\beta)\leq (k,\alpha)$ such that 
\[\lab_2(k+1,\beta)=(\pr_1(\pi_k(t)),\pr_2(\pi_k(t))+m_{k,\alpha}\; (\!\!\!\!\!\!\mod M_{k,\alpha})).\] 
If so, decree $^t\!z^k_\alpha :=x^{k+1}_\beta$. Otherwise, pick any $(k+1,\beta)\leq (k,\alpha)$ with $\rho(x^{k+1}_\beta,x^k_\alpha)<\delta^{k+1}$ and decree $^t\!z^k_\alpha :=x^{k+1}_\beta$. (We could again choose, for example, the nearest child of $x^k_\alpha$ and eliminate the arbitrariness of this choice.)
\end{remark}

\section{Applications}
\subsection{Set-up}
Let $(X,\rho)$ be a quasi-metric space and suppose that $\mu$ is a positive Borel-measure on $X$ satisfying the doubling condition
\begin{equation}\label{def:doubling}
\mu(2B)\leq C\mu(B)\quad\text{for all balls $B$}.
\end{equation}
Note that if $\mu$ satisfies the above doubling condition then $0<\mu(B)<\infty$ for all balls $B$. Let us state the following well-known lemma.

\begin{lemma}\label{prop;balls;measures}
For every $x\in X$ and $0<r\leq R$ we have
\[\frac{\mu(B(x,R))}{\mu (B(x,r))}\leq C_\mu \left(\frac{R}{r}\right)^{c_\mu} \] 
where $c_\mu=\log_2C_\mu$ and $C_\mu$ is the smallest constant satisfying \eqref{def:doubling}.
\end{lemma}
There is an immediate sequel to dyadic cubes: 
\begin{corollary}\label{cor;balls;measures}
There exists a constant $C\geq 1$ such that for every dyadic cube $Q$ there holds $\mu (B_Q)\leq C\mu (Q)$, where $B_Q \supseteq Q$ is the containing ball of $Q$ as in \eqref{eq:contain}. Conversely, given a ball $B:=B(x,r)$, there exists a dyadic system $\mathscr{D}^t$ and a dyadic cube $Q_B\in \mathscr{D}^t$ such that $B\subseteq Q_B$ and $\mu(Q_B)\leq C\mu(B)$.
\end{corollary}
\begin{proof}
Given a dyadic cube $Q$ and a ball $B(x,r)$, consider the balls $B(x_{\alpha}^k, c_1\delta^k)\subseteq Q^k_\alpha \subseteq B(x_{\alpha}^k, C_1\delta^k)=:B_Q$ from Theorem \ref{thm:cubes} and the cube $Q_B$ from Lemma~\ref{prop;cubeandball} with $B(x,r)\subseteq Q_B \subseteq B(x, Cr)$. The assertion follows readily from Lemma~\ref{prop;balls;measures}.
\end{proof}

\subsection{Maximal operators}
Let $\omega$ be a weight on $X$, i.e. $\omega\geq 0$ and $\omega\in L^1_{\loc}(X,\mu)$. Given a measurable set $E$, denote $\omega (E):=\int_{E}\omega\,d\mu$. Define \textit{weighted Hardy--Littlewood maximal operator} $M_\omega$ by
\begin{equation}\label{operator}
M_\omega f(x)=\sup_{B\ni x}\frac{1}{\omega(B)}\int_{B}\abs{f}\,\omega d\mu, \quad f\in L^1_{\loc}(X,\mu), x\in X,
\end{equation}
where the supremum is over all balls $B$ containing $x$. We drop the subscript $\omega$ in $M_\omega$ if $\omega\equiv 1$. 

Given a weight $\omega$ and $p>1$, set $\sigma =\omega^{-1/(p-1)}$. We say that $\omega$ satisfies the $A_p$-condition and denote $\omega\in A_p$ if
\begin{equation*}
\| \omega \|_{A_p} := \sup_{B}\frac{\omega(B)\sigma(B)^{p-1}}{\mu(B)^p}<\infty .
\end{equation*}
Note that $\omega\in A_p$, if and only if $\sigma\in A_{p'}$ where $1/p+1/p'=1$.

By the fundamental result of B. Muckenhoupt, the classical Hardy--Littlewood maximal operator $M$ is bounded on $L^p_\omega$, $1<p<\infty$, if and only if $\omega\in A_p$. In this section, we will provide a quantitative formulation of this well-known result on the metric space $(X,\mu)$.

Let $\mathscr{D}^t$ denote any fixed dyadic grid of cubes $Q^k_\alpha, k\in \Z,\alpha\in \mathscr{A}_k$ constructed in Section~\ref{sec:adjacent}. We define the \textit{weighted dyadic maximal operator} $M^{\mathscr{D}^t}_\omega$ by
\begin{equation}\label{def;dyadicoperator}
M^{\mathscr{D}^t}_\omega\!\! f(x)=\sup_{Q\ni x}\frac{1}{\omega(Q)}\int_{Q}\abs{f}\, \omega d\mu , \quad f\in L^1_{\loc}(X,\mu), x\in X ,
\end{equation}
where the supremum is over all dyadic cubes $Q\in \mathscr{D}^t$ containing $x$.

By a well-known fact, see for example \cite{williams}, Theorem 14.11, the dyadic maximal operators $M^{\mathscr{D}^t}_\omega$ are bounded on $L^p_\omega$, $1<p<\infty$, uniformly in all weights $\omega$: For $1<p<\infty$ there holds
\begin{equation}\label{sharp:ineq;dyadic}
\| M^{\mathscr{D}^t}_\omega \!\! f\|_{L^p_\omega}\leq p'\, \| f\|_{L^p_\omega} 
\end{equation}
for all $f\in L^p_\omega$.

We further consider the closely related \textit{sharp maximal operator} $M^\#$ defined by
\[M^\# f(x):=\sup_{B \ni x}\frac{1}{\mu(B)}\int_{B}\abs{f-f_B}\, d\mu, \quad f\in L^1_{\loc}(X,\mu), x\in X, \]
where the supremum is over all balls containing $x$, and we have used the notation
\[f_E:=\frac{1}{\mu(E)}\int_{E}\abs{f}\,d\mu, \]
for the integral average of $f$ over a bounded measurable set $E$, $\mu(E)>0$. Further define the \textit{dyadic sharp maximal operator} $M^\#_{\mathscr{D}^t} f$ by
\[M^\#_{\mathscr{D}^t} f(x):=\sup_{Q \ni x}\frac{1}{\mu(Q)}\int_{Q}\abs{f-f_Q}\, d\mu, \quad f\in L^1_{\loc}(X,\mu), x\in X, \]
where the supremum is over all dyadic cubes $Q\in \mathscr{D}^t$ containing $x$.

We will first show the equivalence between the classical operators $M$ and $M^\#$ and their dyadic counterparts:
\begin{proposition}\label{prop:equiv;operators}
Let $f\in L^1_{\loc}(X,\mu)$. We have the pointwise estimates
\begin{align}
M^{\mathscr{D}^t}\!\!f(x)\leq CMf(x)\qquad\text{and}\qquad 
Mf(x) \leq C\sum_{t=1}^{K}M^{\mathscr{D}^t}\!\!f(x); \label{lemma;one}\\
M^\#_{\mathscr{D}^t}f(x)\leq CM^\# f(x)\qquad\text{and}\qquad 
M^\# f(x) \leq C\sum_{t=1}^{K}M^\#_{\mathscr{D}^t}f(x) \label{lemma;two}
\end{align}
where both of the first inequalities hold for every $t=1,\ldots ,K$ and the constant $C\geq 1$ is independent of $f$.
\end{proposition}

\begin{proof}
Fix $t$ and assume $x\in Q, Q\in \mathscr{D}^t$. Let $B\supseteq Q$ be the containing ball of $Q$ as in \eqref{eq:contain}. Then $\mu(B)\leq C\mu(Q)$ by Lemma~\ref{cor;balls;measures}, and thereby
\begin{align*}
\frac{1}{\mu(Q)}\int_{Q}\abs{f-f_Q}\,d\mu 
& \leq\frac{1}{\mu(Q)}\int_{Q}\abs{f-f_B}\,d\mu +\left\vert f_B -f_Q\right\vert\\
& \leq \frac{2}{\mu(Q)}\int_{Q}\abs{f-f_B}\,d\mu \leq \frac{2C}{\mu(B)}\int_{B}\abs{f-f_B}\,d\mu  \\
& \leq 2C M^\# f(x).
\end{align*}
The first inequality in \eqref{lemma;two} follows by taking a supremum over all dyadic cubes in $\mathscr{D}^t$ containing $x$. Also the first inequality in \eqref{lemma;one} follows by putting $f_B=f_Q=0$ in the above and making the obvious simplifications.

For the reverse inequalities, consider ball $B\ni x$ and let $Q=Q(t)$ be the dyadic cube as in  Lemma~\ref{prop;cubeandball} with $B\subseteq Q$ and $\mu(Q)\leq C\mu(B)$. By repeating the argumentation above with the roles of $B$ and $Q$ interchanged we may conclude with
\begin{align*}
\frac{1}{\mu(B)}\int_{B}\abs{f-f_B}\,d\mu & \leq 2C M^\#_{\mathscr{D}^t}f(x)\leq 2C\sum_{t=1}^{K} M^\#_{\mathscr{D}^t}f(x) .
\end{align*}
The second inequality in \eqref{lemma;two} follows again by taking a supremum over all balls containing $x$. Also the second inequality in \eqref{lemma;one} follows as earlier.
\end{proof}

\subsection{The sharp weighted norm of the Hardy--Littlewood maximal operator}
As to illustrate the use of the new adjacent dyadic systems constructed in Section~\ref{sec:adjacent}, we will provide an easy extension of the Buckley's theorem \cite{Buckley:93} on the sharp dependence of $\| M\|_{L^p_\omega}$ on $\| \omega\|_{A_p}$ in Muckenhoupt's theorem for Hardy--Littlewood maximal function to metric spaces $X$ with a doubling measure $\mu$:

\begin{proposition}
Let $1<p<\infty$. Then 
\begin{equation*}
  \Norm{Mf}{L^p_w}\leq C\Norm{w}{A_p}^{1/(p-1)}\Norm{f}{L^p_w},
\end{equation*}
where the constant $C$ depends only on $X,\mu$ and $p$.
\end{proposition}

\begin{proof}
By Proposition~\ref{prop:equiv;operators}, it suffices to prove the analogous estimate for the dyadic maximal operator $M^{\mathscr{D}^t}\!\!$. We follow the Euclidean approach due to Lerner~\cite{Lerner:08} with the deviation that we will utilize the universal maximal function estimate \eqref{sharp:ineq;dyadic} instead of the corresponding result for the centred maximal operator which, in the Euclidean case, is a well-known consequence of the Besicovich covering theorem --- a powerful classical tool generally unavailable in abstract metric spaces. We repeat the details for the reader's convenience.

Suppose $0\leq \omega\in L^1_{\loc}(X,\mu)$. Denote $\sigma =\omega^{-1/(p-1)}$ and
\[A_p(Q):=\frac{\omega(Q)\sigma(Q)^{p-1}}{\mu(Q)^p}. \]
Fix $t$ and suppose $x\in Q, Q\in\mathscr{D}^t$. We have
\begin{align*}
\frac{1}{\mu(Q)}\int_{Q}\abs{f}\, d\mu &= A_p(Q)^{\frac{1}{p-1}}\left[ \frac{\mu (Q)}{\omega(Q)}\left( \frac{1}{\sigma(Q)}\int_{Q}\abs{f}\, d\mu\right)^{p-1}\right]^{\frac{1}{p-1}}\\
& \leq \| \omega\|_{A_p}^{\frac{1}{p-1}} 
\left[ \frac{1}{\omega(Q)}\int_{Q}\left( \frac{1}{\sigma(Q)}\int_{Q}\abs{f\sigma^{-1}}\, \sigma d\mu\right)^{p-1}d\mu\right]^{\frac{1}{p-1}}\\
& \leq \| \omega\|_{A_p}^{\frac{1}{p-1}} 
\left[ \frac{1}{\omega(Q)}\int_{Q}\Big(M^{\mathscr{D}^t}_\sigma(f\sigma^{-1})(x) \Big)^{p-1}d\mu(x)\right]^{\frac{1}{p-1}}\\
& \leq \| \omega\|_{A_p}^{\frac{1}{p-1}} 
\left[ \frac{1}{\omega(Q)}\int_{Q}\Big([M^{\mathscr{D}^t}_\sigma(f\sigma^{-1})]^{p-1}\omega^{-1}\Big)\omega d\mu\right]^{\frac{1}{p-1}},
\end{align*}
and hence,
\[M^{\mathscr{D}^t}\!\!f(x)\leq \| \omega\|_{A_p}^{\frac{1}{p-1}} 
M^{\mathscr{D}^t}_\omega\Big([M^{\mathscr{D}^t}_\sigma(f\sigma^{-1})]^{p-1}\omega^{-1}\Big)(x)^{\frac{1}{p-1}}. \]
Therefore, since $p/(p-1)=p'$ and $\sigma =\omega^{-1/(p-1)}$, the above estimate together with the universal maximal function estimate \eqref{sharp:ineq;dyadic} implies that
\begin{align*}
\|M^{\mathscr{D}^t}\!\! f \|_{L^p_\omega}& \leq
\| \omega\|_{A_p}^{\frac{1}{p-1}} 
\Big\| M^{\mathscr{D}^t}_\omega\Big([M^{\mathscr{D}^t}_\sigma(f\sigma^{-1})]^{p-1}
\omega^{-1}\Big)\Big\|^{\frac{1}{p-1}}_{L^{p'}_\omega}\\
&\leq \| \omega\|_{A_p}^{\frac{1}{p-1}}  p^{\frac{1}{p-1}}\;
\| [M^{\mathscr{D}^t}_\sigma(f\sigma^{-1})]^{p-1}\omega^{-1}\|_{L^{p'}_\omega}^{\frac{1}{p-1}}\\
& = \| \omega\|_{A_p}^{\frac{1}{p-1}}  p^{\frac{1}{p-1}}\;
\| M^{\mathscr{D}^t}_\sigma(f\sigma^{-1})\|_{L^{p}_\sigma}\\
&\leq \| \omega\|_{A_p}^{\frac{1}{p-1}}  p^{\frac{1}{p-1}}p'\;
\| f\sigma^{-1}\|_{L^{p}_\sigma}\\
& =\| \omega\|_{A_p}^{\frac{1}{p-1}}  p^{\frac{1}{p-1}}p'\;
\| f\|_{L^{p}_\omega}.\qedhere
\end{align*}
\end{proof}


\subsection{Functions of bounded mean oscillation}
Recall that \textit{the classical $\BMO(\mu)$ space} is the set of equivalence classes of functions $f\in L_{\loc}^1(X,\mu)$, modulo additive constants, such that the $L^1$-averages 
\[\frac{1}{\mu(B)}\int_{B} \abs{f-f_B}\, d\mu,\qquad f_B:=\frac{1}{\mu(B)}\int_{B}\abs{f}\,d\mu, \]
are bounded (uniformly in $B$). The non-negative real number
\begin{equation}\label{def;BMO}
\| f \|_{\BMO}:= \underset{B}\sup \frac{1}{\mu(B)}\int_{B} \abs{f-f_B}\, d\mu <\infty,
\end{equation}
where the supremum is over all balls, is then called the $\BMO$-norm of $f$.

For every $t= 1, \ldots , K$, we define a \textit{dyadic $\BMO(\mu)$ space} $\BMO_{\mathscr{D}^t}$ as the set of equivalence classes of functions $f\in L^1_{\loc}(X,\mu)$, modulo additive constants, such that 
\[\|f\|_{\BMO_{\mathscr{D}^t}}:=\underset{Q\in\mathscr{D}^t}\sup\frac{1}{\mu(Q)}\int_{Q} \abs{f-f_Q}\, d\mu<\infty,\qquad f_Q:=\frac{1}{\mu(Q)}\int_{Q}\abs{f}\,d\mu, \]
where the supremum is over all dyadic cubes $Q\in\mathscr{D}^t$. The quantity $\|f\|_{\BMO_{\mathscr{D}^t}}$ is then the dyadic $\BMO$-norm of $f$.

The relationship between the two kinds of $\BMO$ spaces has been studied in the Euclidean setting in \cite{Garnett} and \cite{Mei03}. 

As a second illustration of the use of the new adjacent dyadic systems, we provide a representation of $\BMO(\mu)$ as an intersection of finitely many dyadic $\BMO(\mu)$ spaces. This extends the Euclidean result, which was explicitly stated by T. Mei \cite{Mei03}, but already implicit in some earlier work; cf. \cite{Mei03}, Remark 6. A related result in metric spaces was also proven by Caruso and Fanciullo \cite{CF}.
\begin{proposition}\label{prop:bmo}
Suppose $(X,\rho)$ is a quasi-metric space and $\mu$ is a positive Borel-measure on $X$ with the doubling property \eqref{def:doubling}. There exist constants $C>0$ and $C'>0$ depending only on $X$ and $\mu$ such that for every $f\in L_{\loc}^1(X,\mu)$, there holds
\[C\|f \|_{\BMO_{\mathscr{D}^t}}\leq \|f \|_{\BMO}\leq C' \sum_{t=1}^{K} \|f \|_{\BMO_{\mathscr{D}^t}} \]
where the first estimate holds for every $t=1,\ldots ,K$. Thus,
\[\BMO(\mu)=\displaystyle\bigcap_{t=1}^{K}\BMO_{\mathscr{D}^t}(\mu)\]
with equivalent norms.
\end{proposition}
\begin{proof}
This is an immediate corollary of Proposition~\ref{prop:equiv;operators}.
\end{proof}

\begin{remark}
It is a well-known fact that both the classical and dyadic $\BMO(\mu)$ spaces satisfy the John--Nirenberg inequality. The proof for the dyadic version is slightly easier. One may represent the space $\BMO^p(\mu)$ as an intersection of finitely many dyadic spaces $\BMO^p_{\mathscr{D}^t}$, $p>1$, as stated in Proposition~\ref{prop:bmo}. With this representation, one may derive the John--Nirenberg inequality and the exponential integrability of $\BMO(\mu)$ functions from their dyadic counterparts, thereby avoiding some technicalities in the proof.
\end{remark}

\bibliographystyle{plain}
\def\cprime{$'$} \def\cprime{$'$} \def\cprime{$'$}

\end{document}